\documentclass[leqno,12pt]{amsart}
\usepackage[thmnumber=subsection,eqnumber=subsubsection]{jdr-style}
\usepackage{jdr-macros,jdr-tikz}
\usepackage{tikz-cd}
\usepackage{longtable}
\usepackage{url}
\usepackage{verbatim}
\usepackage{tikz}
\usepackage{mathtools}
\usepackage{xr}

\newtheorem{art}[subsection]{}

\newcommand{\dpa}{{\d'_{\rm P}}}
\newcommand{\dpb}{{\d''_{\rm P}}}

\title{On semipositive piecewise linear functions in non-archimedean analytic geometry}

\author[W.~Gubler]{Walter Gubler}
\address{W. Gubler, Mathematik, Universit{\"a}t 
	Regensburg, 93040 Regensburg, Germany}
\email{walter.gubler@mathematik.uni-regensburg.de}

\author[J.~Rabinoff]{Joseph Rabinoff}
\address{J. Rabinoff, Department of Mathematics, Trinity College of Arts and Sciences, Duke University, Durham, NC 27708, USA}
\email{jdr@math.duke.edu}

\bibalias{CLD}{chambert_ducros12:forms_courants}

\let\mathbb=\mathbf
\def\artref#1{\ref{#1}}

\tikzcdset{arrow style=math font}

\mathtext{Zar}
\mathtext{Frac}
\mathtext{Int}
\mathtext{PL}
\mathtext{perf}

\mathtext{centdim}
\mathtext{alg}

 \thanks{W.~Gubler
	was supported by the collaborative research 
	center SFB 1085 \emph{Higher Invariants - Interactions between Arithmetic Geometry and Global Analysis} funded by the Deutsche Forschungsgemeinschaft.}

\begin{document}
	
\begin{abstract}
  In this paper, we generalize results on Zhang's semipositive model metrics from the algebraic setting to strictly analytic spaces over a non-trivially valued non-Archimedean field. We prove stability  under pointwise limits and under forming the maximum. We also prove the maximum principle. As tools, we use a local lifting theorem from the special fiber to the generic fiber of a model and an affinoid Bieri--Groves theorem.
  \end{abstract}

\keywords{{Berkovich spaces, tropicalization,  plurisubharmonic functions}} 
\subjclass{{Primary 14G40; Secondary 31C05, 32P05, 32U05}}

\maketitle


\section{Introduction}

In Diophantine Geometry, an important tool is the height of a proper variety $X$ with respect to a given metrized line bundle $L$. This height is defined as an arithmetic self intersection number  introduced by Gillet and Soul\'e \cite{gillet_soule_90} in higher dimensions. Zhang \cite{zhang95} realized that the contribution of a non-Archimedean place $v$ to this arithmetic intersection product is given by a $v$-adic metric associated to a model $(\sX,\sL)$ of $(X,L)$ over the valuation ring of $v$. Such a model metric is called \emph{semipositive} if the model $\sL$ is a nef line bundle on the proper variety $\sX$. More generally,  Zhang introduced \emph{semipositive continuous metrics} as uniform limits of semipositive model metrics. He showed that the canonical metrics of polarized dynamical systems are semipositive continuous metrics. This point of view turned out to be very important in the proof of the Bogomolov conjecture by Ullmo \cite{ullmo98:positivite_manin_mumford} and Zhang \cite{zhang98}. As pointed out in \cite{gubler98:local_heights_subvariet}, the model metrics can be studied over any non-trivially valued non-Archimedean field working with  Berkovich spaces and their formal models leading to a theory of local heights for proper varieties.

Let $k$ be a field that is complete with respect to a nontrivial, non-Archimedean  absolute value. 
The theory of subharmonic functions on a smooth, strictly analytic $k$-curve $X$ has been  worked out by Thuillier in his thesis~\cite{thuillier05:thesis}. In higher dimensions,  we still lack a completely satisfactory theory of plurisubharmonic   functions on Berkovich spaces. As work of Boucksom and Jonsson \cite{BJ22globalpluripotentialtheorytrivially} in the algebraic setting suggests, such a theory might be build upon regularization by semipositive piecewise linear functions. This means that plurisubharmonic functions should be given as decreasing limits of semipositive piecewise linear functions. The latter are the local analogues of semipositive model metrics. The crucial role of semipositive piecewise linear functions is then supported by Zhang's global approach to semipositive continuous metrics. 

In this paper, we study semipositivity of piecewise linear functions on a good strictly $k$-analytic space $X$. This completes our study of such functions in \cite{gubler_rabinoff_jell:harmonic_trop}, and generalizes results given in \cite{gubler_martin19:zhangs_metrics} from the algebraic to the analytic setting. 

\subsection*{Piecewise linear functions}
Let $X$ be a good strictly $k$-analytic space $X$ (see Section~\ref{section: preliminaries} for more details on the prerequisites on non-Archimedean analytic spaces). 
We say that $h\colon X \to \R$ is \emph{piecewise linear} if for any $x \in X$ there are finitely many strictly $k$-analytic domains $V_j$ and invertible analytic functions $g_j$ on $V_j$ such that $\bigcup V_j$ is a neighbourhood of $x$ in $X$ and such that $h=\log |g_j|$ on $V_j$.
An \emph{$\R$-PL function} is a real function which is locally an $\R$-linear combination of piecewise linear functions.  

The semipositivity condition can be best described in terms of Temkin's reduction of germs.  
To any point $x\in X$, Temkin associates~\cite{temkin00:local_properties} the \defi{reduction} $\red(X,x)$ of the germ $(X,x)$; this is a kind of Riemann--Zariski space defined over the residue field~$\td k$ which keeps track of all local formal models over the valuation ring of $k$.  There is a canonical homomorphism $u\mapsto L_u(x)$ from the space of germs of PL functions at $x$ to $\Pic(\red(X,x))$.  This homomorphism extends to a homomorphism from the space of germs of $\R$-PL functions to $\Pic(\red(X,x)) \otimes \R$. We say that an $\R$-PL function $h$ is \emph{semipositive at $x\in X$} if  $L_u(x)$ is nef in $\Pic(\red(X,x)) \otimes \R$.  See \cite[Section 6]{gubler_rabinoff_jell:harmonic_trop} for more details.

\subsection*{Lifting theorems} 
The idea is to relate the notion of numerical effectivity of the line bundle $L_u(x)$ on the reduction, which is a condition that is checked on curves embedded in the \emph{special} fiber of some formal model, to the plurisubharmonicity of $u$, which is a condition that is by definition checked on curves embedded in the \emph{generic} fiber.  This is made possible by a suitable \emph{local lifting theorem}, which allows us locally to lift a subvariety of the special fiber of a formal model to a Zariski-closed subspace of the analytic generic fiber.  This is the technical heart of the paper.  We only describe here in the introduction the affine lifting theorem, given in Theorem~\ref{thm:affine lifting}, which is of global nature.

\begin{thm} \label{intro-thm: affine lifting}
Let $X$ be a strictly $k$-affinoid space with affine formal model $\fX$. Then for every $d$-dimensional irreducible closed subvariety $W$ of the special fibre of $\fX$, there is an irreducible Zariski closed subspace $Y$ of $X$ of dimension $d$ such that $W$ is an irreducible component of the special fibre of the formal closure of $Y$ in $\fX$.
\end{thm}

For flat algebraic (but not necessarily affine) schemes over the valuation ring of $k$, this was shown in \cite[Theorem 4.1]{gubler_kuenneman19:positivity}. In the analytic setting, Theorem \ref{intro-thm: affine lifting} does not imply a global result, as a Zariski-closed subspace of an affinoid domain in an analytic space does not usually extend to the whole space.  This is the reason that we can prove only locally a lifting theorem in case of an arbitrary good quasi-compact strictly $k$-analytic space (Theorem \ref{local lifting theorem}). The argument is independent of Theorem \ref{intro-thm: affine lifting} and is due to Antoine Ducros.

\subsection{Semipositivity and restriction to curves} The local lifting theorem shows that semipositivity of $\R$-PL functions is the analogue of plurisubharmonicity in the complex setting, where psh is defined in terms of restriction to curves. 

\begin{thm} \label{intro-thm: semipositivity and curve restriction}
Let $x$ be a point of  a good strictly analytic space $X$ over $k$ and let $h\colon X \to \R$ be $\R$-PL. Then $h$ is semipositive at $x$ if and only if for every sufficiently small strictly analytic neighbourhood $U$ of $x$ and every Zariski-closed curve $C$ in $U$, the restriction of $h$ to $C$ is a semipositive $\R$-PL function.
\end{thm}

This will be shown in Theorem \ref{semipositivity and curve restriction}. We can then use base change and normalization to reduce semipositivity of $\R$-PL functions to the case of smooth analytic curves (Corollary \ref{restriction to smooth curves}). 
As an application, we deduce immediately from the corresponding result for smooth curves in Thuillier's thesis that the maximum of two semipositive $\Q$-PL functions is again a semipositive $\Q$-PL function (Corollary \ref{maximum of semipositive}). This and the following application of Corollary \ref{restriction to smooth curves} were shown in \cite{gubler_martin19:zhangs_metrics} only for proper algebraic schemes $X$ over $k$.

\begin{thm} \label{intro-thm: pointwise convergence of semipositive PL}
Let $X$ be a good strictly $k$-analytic space of $k$. We assume that the sequence $(h_j)_{j \in N}$ of semipositive $\R$-PL functions converges pointwise to an $\R$-PL function $h$. Then $h$ is semipositive.
\end{thm}

This result, shown in Theorem \ref{pointwise convergence of semipositive PL},  gives also a new proof in the algebraic case. 

We say that an $\R$-PL function $h\colon X \to \R$ is a \defi{locally Fubini--Study} function if the residue $L_{h}(x)$ is semiample in $\Pic(\red(X,x))\otimes \R$ for every $x \in X$. It is clear that locally Fubini--Study functions are semipositive $\R$-PL functions. We  can then prove the following local regularization result. 

\begin{thm} \label{intro-thm: semipositive and LFS}
Let $h\colon X\to \R$ be an $\R$-PL function on a boundaryless  strictly $k$-analytic space $X$ and let $x \in X$. Then $h$ is semipositive at $x$ if and only if there is a strictly analytic neighbourhood $U$ of $x$ and locally Fubini--Study functions $(h_j)_{j \in \N}$ on $U$ converging uniformly to $h|_U$. 
\end{thm}

The only if direction is shown in \cite[Proposition 7.10]{gubler_rabinoff_jell:harmonic_trop} and the converse follows from Theorem \ref{intro-thm: pointwise convergence of semipositive PL}.

\subsection*{Affinoid Bieri--Groves theorem} Recall that the Bieri--Groves theorem is a crucial result in tropical geometry showing that for a closed $d$-dimensional subvariety $Y$ of the multipicative torus $\mathbb G_{\rm m}^n$, the associated tropical variety $\Trop(Y)$ is a finite union of $d$-dimensional $(\Z,\Gamma)$-polyhedra in $\R^n$ (see \artref{subsection: polyhedral geometry} for the conventions in polyhedral geometry), where $\Gamma=v(k^\times)$ is the value group of the valuation $v=-\log|\phantom{a}|$ of $k$. For a quasicompact strictly analytic space $X$ of pure dimension $d$ and a morphism $\varphi \colon X \to \mathbb G_{\rm m}^{n,\an}$, we define a \defi{smooth tropicalization map} $\varphi_{\rm trop}\colon X \to \R^n$ by $\trop\circ \varphi$ and the associated tropical variety by $\varphi_{\rm trop}(X)$. 
By results of Berkovich and Ducros (see \cite{ducros03:image_reciproque}),   this tropical variety is analogously a finite union of $(\Z,\Gamma)$-polytopes of dimension $\leq d$, but the inequality might be strict. Ducros \cite[Th\'eor\`eme 3.4]{ducros12:squelettes_modeles} has shown that $\trop(\partial X)$ is of  dimension $<d$. In the affinoid case, the results of Ducros yield the following more precise fact which we will show in Theorem \ref{affinoid Bieri-Groves theorem}.

\begin{thm} \label{intro-thm: affinoid Bieri-Groves theorem}
Let $X$ be a strictly affinoid space over $k$ of pure dimension $d$ with relative boundary $\partial X$ and let $\varphi \colon X \to \mathbb G_{\rm m}^{n,\an}$ be a morphism. Then the associated tropical variety $\varphi_{\rm trop}(X)$ has pure dimension $d$ at every point of $\varphi_{\rm trop}(X) \setminus \varphi_{\rm trop}(\partial X)$.
\end{thm}

\subsection*{Maximum principle}
We will study the maximum principle in the class of piecewise smooth functions. Let $X$ be a separated good strictly $k$-analytic space. A function $u\colon X \to \R$ is called \defi{piecewise smooth} if there is a $\G$-covering of $X$ by compact strictly $k$-analytic domains $(V_i)_{i \in I}$ such that for every $i \in I$, there is a smooth tropicalization map $\varphi_{i,\rm trop}\colon V_i \to \R^n$ as introduced above and a smooth function $f_i\colon \R^n \to \R$ such that $u|_{V_i}=f_i \circ \varphi_{i,\rm trop}$ on $V_i$. We note that $\R$-PL functions are the special case of piecewise smooth functions where all the $f_i$'s can be chosen as affine functions. 

In the theory of Chambert-Loir and Ducros~\cite[\S 5.4]{chambert_ducros12:forms_courants}, we can associate to every continuous real function $u$ on $X$ an associated current $[u]$. The function $h$ is called \defi{plurisubharmonic} (or just \defi{psh}) \defi{in the sense of Chambert-Loir and Ducros} if the associated current $\d'\d''[h]$ is positive.  See \cite[\S 5.5]{chambert_ducros12:forms_courants} for details. Then we can show the following \emph{maximum principle.}

\begin{thm}\label{intro-thm: maximum.principle}
  Let $X$ be a good, separated, strictly $k$-analytic space, and let $u$ be a  piecewise smooth function on $X$ which is psh in the sense of Chambert-Loir and Ducros.  If $u$ has a local maximum on $x\in X \setminus \partial X$, then $u$ is locally constant at~$x$.
\end{thm}

Since any semipositive $\R$-PL function is psh in the sense of Chambert-Loir and Ducros \cite[Theorem 7.14]{gubler_rabinoff_jell:harmonic_trop}, the maximum principle holds for semipositive $\R$-PL functions. We prove Theorem~\ref{intro-thm: maximum.principle} by proving a tropical analogue in Theorem \ref{tropical maximum principle}: namely, if $C\subset\R^n$ is a tropical variety of pure dimension $d$ and if $f\colon C\to\R$ is \defi{weakly tropically psh}, in the sense that the Lagerberg current $\d'\d''[f]$ is positive, then $f$ satisfies a similar maximum principle.  The proof uses the Bieri--Groves theorem for affinoid varieties given in Theorem \ref{intro-thm: affinoid Bieri-Groves theorem}. Then Theorem \ref{intro-thm: maximum.principle} is shown in Theorem \ref{thm: maximum principle for semipositive functions}.

\begin{rem}
  We do not know whether an $\R$-PL function that is psh in the sense of Chambert-Loir and Ducros is necessarily semipositive.  This is strongly related to the question of whether psh functions are stable under pullback, which is likewise unknown.
\end{rem}

\subsection*{Acknowledgments}

We are very grateful to Antoine Ducros who helped us with the argument for the local lifting theorem. We thank Mattias Jonsson for comments about an earlier draft of the paper.

\section{Preliminaries} \label{section: preliminaries}

\subsection{General Notation and Conventions} \label{general notation and conventions}
Throughout this paper, $k$ will denote a field that is complete with respect to a nontrivial, non-Archimedean valuation.

The set of natural numbers $\N$ includes $0$. We allow equality in a set-theoretic inclusion $S \subset T$.  
If $P$ is any abelian group, we let $P_\Lambda \coloneqq P \otimes_\Z \Lambda$ be the tensor product with $\Lambda$. For a field $F$, we denote by $\overline F$ an algebraic closure of $F$.

For a ring $A$, the group of invertible elements is denoted by $A^\times$.
A \defi{variety} over a field $F$ is an integral scheme which is of finite type and separated over $\Spec\, F$.

A topological space is called \defi{compact} if it is quasi-compact and Hausdorff. 

\subsection{Non-Archimedean geometry}  \label{Non-Archimedean geometry}

A \defi{non-Archimedean field} is a field $k$ endowed with a  non-Archimedean complete absolute value $|\phantom{a}|\colon k\to \R$. The corresponding valuation ring is denoted by $k^\circ$, its unique maximal ideal by $k^{\circ\circ}$ and its residue field by $\td k=k^\circ/k^{\circ\circ}$.
An \defi{analytic extension field} is a non-Archimedean field $k'$ equipped with an isometric embedding $k\inject k'$.

We will use only good Berkovich analytic spaces over $k$ in the terminology from~\cite{berkovic93:etale_cohomology}. These are the $k$-analytic spaces considered in~\cite{berkovic90:analytic_geometry}. Here, \defi{good} means that every  point of the analytic space $X$ has a neighbourhood isomorphic to the Berkovich spectrum $\sM(\sA)$ of an affinoid algebra $\sA$. We denote by $\sO_X$ the sheaf of analytic functions on a good analytic space $X$ and by 	$\sH(x)$ the completion of the residue field  $\sO_{X,x}/\fm_{X,x}$. In this paper, we are mainly concerned with good \defi{strictly analytic} spaces, which means that we can choose $\sA$ above isomorphic to a quotient of a Tate algebra $k\angles{x_1,\ldots,x_n}$. Here, we recall that 
$$k\angles{x_1,\ldots,x_n} \coloneqq \left\{\sum_m a_m x^m \in k\ps{x_1,\dots,x_n}\mid \lim_{|m| \to \infty} |a_m|=0\right\}$$
endowed with its Gauss norm. Similarly, we define the subring $k^\circ\angles{x_1,\ldots,x_n}$ assuming all $a_m \in k^\circ$ in the above definition. For a (strictly) affinoid algebra $\sA$, the subring of power bounded elements is denoted by $\sA^\circ$. The Berkovich spectrum of the Tate algebra $k\angles{x_1,\ldots,x_n}$ is the \defi{closed unit ball of dimension $n$}  denoted by $\bB^n$.

For a good strictly $k$-analytic space $X$ over $k$, we also consider  the  $\G$-topology  generated by the strictly affinoid domains. This is most natural for comparison to rigid analytic geometry: see~\cite[1.6]{berkovic93:etale_cohomology}. 

An important tool is the \defi{relative boundary} of a morphism $\varphi\colon X \to Y$ of good $k$-analytic spaces. This is a closed subset of $X$ which we denote by $\partial(X/Y)$.  See \cite[\S 2.5 and \S 3.1]{berkovic90:analytic_geometry} for the definition and properties. The complement $\Int(X/Y)\coloneqq X \setminus \partial(X/Y)$ is called the \defi{relative interior}. Most often we will apply this in the case $Y=\sM(k)$, in which case we just set $\partial X \coloneqq \partial(X/\sM(k))$. If $\partial X = \emptyset$, then we call $X$ \emph{boundaryless}.

For a scheme $Y$ of finite type over $k$, we denote by $Y^\an$ the \emph{Berkovich analytification} of $Y$. This is a (good) boundaryless strictly $k$-analytic space.  See \cite[\S 3.5]{berkovic90:analytic_geometry} for details.

\subsection{Models and formal geometry} \label{subsection: models and formal geometry}

Recall that a quasi-compact \emph{admissible} formal scheme  over $k^\circ$ is a quasi-compact  formal scheme $\fX$ over $k^\circ$ such that $\fX$ has an open covering by affine formal schemes $\Spf(A)$ with $A$ a  $k^\circ$-algebra which is topologically of finite type and flat over $k^\circ$. The \emph{generic fiber} of $\fX$ is defined on affines by $\Spf(A)_\eta \coloneqq \sM(A \otimes_{k^\circ} k)$, and the \defi{special fibre} is defined on affines by $\Spf(A)_s \coloneqq \Spec(A \otimes_{k^\circ} \td k)$. There is a canonical \defi{reduction map} $\red\colon \fX_\eta \to \fX_s$ which is surjective, anti-continuous, and functorial \cite[Proposition 2.17]{gubler_rabinoff_werner:tropical_skeletons}. 

\begin{defn} \label{formal models}
  Let $X$ be a compact strictly $k$-analytic space.  A \defi{formal $k^\circ$-model} of $X$ is a quasi-compact admissible formal scheme $\fX$ over $k^\circ$ with an identification $\fX_\eta=X$. For details, we refer to \cite{bosch14:lectures_formal_rigid_geometry}.
	
	If $L$ is a line bundle on $X$, then a {formal $k^\circ$-model} $(\fX,\fL)$ of $(X,L)$ is a formal $k^\circ$-model $\fX$ of $X$ and a line bundle $\fL$ on $\fX$ with an identification $\fL|_X= L$.  
\end{defn}

A model $(\fX,\fL)$ of $(X,\sO_X)$ gives rise to a metric $\metr_\fL$ on $\sO_X$ by requiring that for any formal open subset $\fU$ of $\fX$ and for any frame $s$ of $\fL$ over $\fU$, we have $\|s\|_\fL=1$ on $\fU_\eta$. Then we have an associated (continuous) function $h_\fL \coloneqq -\log\|1\|_ \fL$ on $X$.

Let $\fX$ be a formal $k^\circ$-model of $X$. The isomorphism classes of line bundles $\fL$ on $\fX$ equipped with an identification $\fL|_X=\sO_X$ form an abelian group denoted by $M(\fX)$. We call the elements of the group $M(\fX)_\Lambda \coloneqq M(\fX)\otimes_\Z \Lambda$ \defi{models of $\sO_X$ with coefficients in $\Lambda$}. For $\fL \in M(\fX)$, we define the metric $\metr_\fL$ and the function $\fL$ by using the above and $\Lambda$-linearity. 

For generalization of these notions to a paracompact good strictly $k$-analytic space $X$, we refer to \cite[\S 2.4, 5.4, 5.6]{gubler_rabinoff_jell:harmonic_trop}.

\subsection{PL functions} \label{subsection: PL functions}

Let $k$ be a {non-trivially valued}  non-Archimedean field. 
We fix an additive subgroup $\Lambda$ of $\R$ which is either divisible or equal to $\Z$. We consider a good, strictly $k$-analytic space $X$. 
A function $h\colon X \to \R$ is called \emph{$\Lambda$-piecewise linear} or \emph{$\Lambda$-PL} if $\G$-locally we have $h=\sum_j \lambda_j \log |f_j|$ for finitely many $\lambda_j \in \Lambda$ and invertible analytic functions $f_j$.  Equivalently, every $x \in X$ has a compact strictly analytic neighborhood $U$, a formal model $\fU$ of $U$, and a formal model $\fL$ of $\sO_U$ on $\fU$ with coefficients in $\Lambda$, such that $h|_U = h_\fL \coloneq -\log\|\scdot\|_\fL$.  See \cite[Proposition~5.8]{gubler_rabinoff_jell:harmonic_trop}.

\subsection{Polyhedral geometry} \label{subsection: polyhedral geometry}

Let $R$ be a subring of $\R$ an let  $\Gamma$ be an $R$-submodule  of $\R$. We say that a function $f \colon \R^n \to \R$ is \defi{$(R,\Gamma)$-linear} if $f$ is an affine function of the form
$$f(x)=m_1 x_1 + \dots m_n x_n + c$$
with $m_1,\dots,m_n \in R$ and $c \in \Gamma$. A \defi{polyhedron} $\Delta$ is a finite intersection of half-spaces $H_i \coloneqq \{x \in \R^n \mid f_i(x) \geq 0\}$ for affine functions $f_i$. If  all $f_i$ can be chosen to be $(R,\Gamma)$-linear, we call $\Delta$ an \defi{$(R,\Gamma)$-polyhedron}.  A \defi{face} of $\Delta$ is the intersection of $\Delta$ with the boundary of a half-space $H \supset \Delta$. We allow also $\emptyset$ and $\Delta$ as faces of $\Delta$. 
A bounded polyhedron is called a \defi{polytope}. 

A \defi{polyhedral complex} is a set $\mathcal C$ of polyhedra in $\R^n$ satisfying the following axioms:
\begin{enumerate}
	\item If $\Delta \in \mathcal C$, then all faces of $\Delta$ are in $\mathcal C$.
	\item If $\Delta, \Delta' \in \mathcal C$, then $\Delta \cap \Delta'$ is a face of $\Delta$ and of $\Delta'$.
\end{enumerate}

\section{The affine lifting theorem}  \label{section:affine lifting theorem}
	
We fix a field $k$ equipped with a nontrivial, non-Archimedean valuation $\val\colon k\to\R\cup\{\infty\}$ and associated absolute value $|\scdot| = \exp(-\val(\scdot))$.

Let $\fX = \Spf(A)$ be an affine admissible formal scheme over $k^\circ$ with generic fiber $X = \sM(\sA)$, where $\sA = A\tensor_{k^\circ} k$.   Set $n=\dim(X)$, and write $A_s = A\tensor_{k^\circ}\td k$ and $\fX_s = \Spec(A_s)$ for the special fiber of $\fX$.

Let $Y\subset X$ be a Zariski-closed subspace, defined by an ideal $I\subset\sA$.  The \defi{closure} of $Y$ in $\fX$ is the Zariski-closed formal subscheme $(\bar Y)^\fX = \Spf(A/(A\cap I))\subset\fX$.  This is $k^{\circ\circ}$-torsionfree, so by~\cite[Proposition~1.1(c)]{BL1}, it is again an admissible formal scheme over $k^\circ$.  Note that the analytic generic fiber of $(\bar Y)^\fX$ is $Y$. We refer to \cite[Proposition 3.3]{gubler98:local_heights_subvariet} for details.

\begin{thm}[Affine Lifting Theorem]\label{thm:affine lifting}
  Let $\fX$ be an affine admissible formal scheme over $k^\circ$ with generic fiber $X$.  Let $W\subset\fX_s$ be an irreducible Zariski-closed subset of dimension $d$.  Then there is a pure-dimensional Zariski-closed subspace $Y\subset X$ of dimension $d$ such that $W$ is an irreducible component of the special fiber of $(\bar Y)^\fX$. 
\end{thm}

For flat algebraic schemes over $k^\circ$, this was shown in \cite[Theorem 4.1]{gubler_kuenneman19:positivity}.

Under the hypotheses of Theorem \ref{thm:affine lifting}, we let $\fY = (\bar Y)^\fX$ for any pure dimensional Zariski-closed subset $Y$.  Then the special fiber $\fY_s$ is pure-dimensional of the same dimension as $Y$ by~\cite[Proposition~2.7(3)]{gubler_rabinoff_werner:tropical_skeletons}. If  $W$ is an irreducible component of $\fY_s$, it follows that $\dim(Y)$ is automatically equal to $\dim(W)$.

  The proof of Theorem \ref{thm:affine lifting} proceeds by induction on $n=\dim(X)$, the case $\dim(X)=0$ being trivial.  If $\dim(X) = \dim(W)$ then we are done, so assume $n > d$.  Passing to an irreducible component of $X$ and its closure in $\fX$, we will assume from now on that $\sA$ is an integral domain.

\begin{art}[Noether Normalization]\label{sec:noether-normalization}
  Since $A$ is an admissible $k^\circ$-algebra, there exists a surjective homomorphism $k^\circ\angles{x_1,\ldots,x_N}\surject A$.  By Noether Normalization~\cite[Theorem~6.1.2/1(i)]{bosch_guntzer_remmert84:non_archimed_analysis}, after composing with an automorphism of $k\angles{x_1,\ldots,x_N}$, we may assume that the homomorphism $k\angles{x_1,\ldots,x_n}\to\sA$ is finite and injective ($n\leq N$).  Any automorphism of $k\angles{x_1,\ldots,x_N}$ induces an automorphism of $k^\circ\angles{x_1,\ldots,x_N} = k\angles{x_1,\ldots,x_N}^\circ$, so this restricts to a homomorphism $f\colon k^\circ\angles{x_1,\ldots,x_n}\inject A$, which is integral by~\cite[Theorem~6.3.5/1]{bosch_guntzer_remmert84:non_archimed_analysis} and the fact that $A\subset\sA^\circ$.  The homomorphism on special fibers
  \[ f_s\colon \td k[x_1,\ldots,x_n] \To A_s \]
  is an integral homomorphism of finitely generated $\td k$-algebras, and is thus finite.

  Let $J\subset A_s$ be the ideal defining $W\subset\fX_s$.  By~\cite[Lemma~10.115.4]{stacks-project} and the argument in its proof based on~\cite[Lemmas~10.115.2 and~10.115.3]{stacks-project}, there is a $\Z$-automorphism $g$ of $\Z[x_1,\dots,x_n]$ such that, setting $y_i = g(x_i)$, the homomorphism
  \[ \td k[y_1,\ldots,y_d] \To A_s/J \]
  is finite and injective (take $S = \td k[x_1,\ldots,x_n]/f_s\inv(J)\subset A_s/J$ in \emph{ibid}). 
   The map $x_i\mapsto g(x_i)$ is an automorphism of $k^\circ\angles{x_1,\ldots,x_n}$, so we may replace the $x_i$ by the $y_i$ to assume:
  \begin{itemize}
  \item The homomorphism $f_\eta\colon k\angles{x_1,\ldots,x_n}\to\sA$ is finite and injective.
  \item The homomorphism $f_s\colon\td k[x_1,\ldots,x_n]\to A_s$ is finite.
  \item The restriction $\td k[x_1,\ldots,x_d]\to A_s/J$ is finite and injective.
  \end{itemize}
  We write
  \[ \Psi\colon \fX\to\fB^n = \Spf(k^\circ\angles{x_1,\ldots,x_n}) \]
  for the finite, surjective morphism corresponding to $f$, and
  \[ \Phi\colon \fX\to\fB^d = \Spf(k^\circ\angles{x_1,\ldots,x_d}) \]
  for the composition of $\Psi$ with the projection $\fB^n\to\fB^d$ onto the first $d$ factors, so we have a commutative triangle
  \begin{equation} \label{eq:Phi.Psi}
    \begin{tikzcd}
      \fX \arrow[rr, "\Psi"] \drar["\Phi"'] & & \bB^n \dlar["\text{projection}"] \\
      & {\bB^d\rlap.} 
    \end{tikzcd}
  \end{equation}
The special fiber of $\Phi$ restricted to $W$ corresponds to the finite, injective homomorphism $\td k[x_1,\ldots,x_d]\to A_s/J$, hence gives a finite, surjective homomorphism
  \[ \Phi_s|_W\colon W\surject\bA^d_{\td k}. \]
\end{art}

\begin{art}[Analytic Fibers]\label{sec:fibers}
  Let $\xi_d$ (resp.\ $\xi_n$) be the Gauss point of $\bB^d = (\fB^d)_\eta$ (resp.\ $\bB^n = (\fB^n)_\eta$).   The only preimage of the generic point $\eta_d$ of $\bA^d_{\td k}$ under the reduction map $\red\colon\bB^d\to\bA^d_{\td k}$ is the Gauss point $\xi_d$ by~\cite[Proposition~2.16(2)]{gubler_rabinoff_werner:tropical_skeletons}.  The reduction map is functorial and surjective by~\cite[Proposition~2.16(1)]{gubler_rabinoff_werner:tropical_skeletons}, so we have a commutative square with surjective arrows indicated:
  \[
    \begin{tikzcd}
      X \rar["\mathrm{red}", two heads] \dar["\Phi_\eta"'] & \fX_s \dar["\Phi_s", two heads] \\
      \bB^d \rar["\mathrm{red}"', two heads] & \bA^d_{\td k}
    \end{tikzcd}
  \]
  Hence we have equalities
  \[ X_{\xi_d} \coloneq \Phi_\eta\inv(\xi_d) = \Phi_\eta\inv(\red\inv(\eta_d))
    = \red\inv(\Phi_s\inv(\eta_d)).
  \]
  The analytic fiber $X_{\xi_d}$ is a strictly $\sH(\xi_d)$-analytic space, and it carries the induced topology from the inclusion $X_{\xi_d}\subset X$ by~\cite[\S1.4, p.30]{berkovic93:etale_cohomology}.

  Since $\Phi_s|_W$ is finite and dominant, we have $\Phi_s\inv(\eta_d)\cap W = \{\eta_W\}$ for the generic point $\eta_W$ of $W$; taking inverse images under reduction gives
  \[ \red\inv(\eta_W) = \red\inv(\Phi_s\inv(\eta_d))\cap\red\inv(W) = X_{\xi_d}\cap\red\inv(W). \]
  By anticontinuity of the reduction map \cite[Propostion 2.17]{gubler_rabinoff_werner:tropical_skeletons}, this is a nonempty open subset of $X_{\xi_d}$.
\end{art}

\begin{art}[Coordinates on $\bB^n_{\xi_d}$]
  Consider the fiber $\bB^n_{\xi_d}$ over $\xi_d$ of the projection $\bB^n\to\bB^d$.  This is a strictly $\sH(\xi_d)$-analytic space that is naturally identified with the ball $\bB^{n-d}_{\sH(\xi_d)}$, and $\Psi_\eta\colon X\to\bB^n$ restricts to a finite morphism $\Psi_{\xi_d}\colon X_{\xi_d}\to\bB^n_{\xi_d}$.  There is a ``coordinate map''
  \[
    \bar{\sH(\xi_d)}{}^{n-d} \To
    \Max\bigr(\sH(\xi_d)\angles{x_{d+1},\ldots,x_n}\bigl) \,\To\,
    \bB^{n-d}_{\sH(\xi_d)}
  \]
  whose image is the set of rig-points.  Concretely, the rig-point $y\in\bB^{n-d}_{\sH(\xi_d)}$ with coordinates $(\alpha_{d+1},\ldots,\alpha_n)\in\bar{\sH(\xi_d)}{}^{n-d}$ is defined by the semi-norm $f\mapsto|f(\alpha_{d+1},\ldots,\alpha_n)|$ on $\sH(\xi_d)\angles{x_{d+1},\ldots,x_n}$.

  Let $Q = \Frac(k\angles{x_1,\ldots,x_d})$, so that $\sH(\xi_d)$ is the completion of $Q$ with respect to the Gauss norm.  The algebraic closure of $Q$ in $\bar{\sH(\xi_d)}$ is denoted by $\bar Q$. 
  We say that a rig-point $y\in \bB^{n-d}_{\sH(\xi_d)}$ has \defi{coordinates in $\bar Q$} if there exists $(\alpha_{d+1},\ldots,\alpha_n)\in\bar Q{}^{n-d}$ mapping to $y$ under the coordinate map.
\end{art}

\begin{lem}\label{lem:coordinate.density}
  Let $\Omega$ be an nonempty open subset of $X_{\xi_d}$.  Then there exists a rig-point $y\in\Psi_{\xi_d}(\Omega)$ with coordinates in $\bar Q$.
\end{lem}

\begin{proof}
  Recall that we are assuming $\sA$ to be an integral domain.  We say that the finite morphism $\Psi_\eta\colon X\to\bB^n$ is \defi{flat} at a point $x\in X$ provided that $\sO_{X,x}$ is a flat $\sO_{\bB^n,\Psi_\eta(x)}$-algebra.  The non-flat locus $T\subset X$ is Zariski-closed by~\cite[Proposition~3.2.8]{berkovic93:etale_cohomology}.  Arguing as in~\artref{sec:fibers} and using the surjectivity of $\Psi_s\colon\fX_s\to\bA^n_{\td k}$, we see that the fiber $X_{\xi_n} \coloneq \Psi_\eta\inv(\xi_n)$ is nonempty.  Since $\xi_n$ is an Abhyankar point of the reduced analytic space $\bB^n$, the local ring $\sO_{\bB^n,\xi_n}$ is a field by~\cite[Example~3.2.10]{ducros18:families}.  It follows that $T$ is disjoint from $X_{\xi_n}$.  Let $B = \Psi_\eta\inv(\Psi_\eta(T))$.  Since $\Psi_\eta$ is finite and continuous with respect to the Zariski topologies, this is a Zariski-closed subset of $X$ containing $T$ and disjoint from $X_{\xi_n}$.  The subset $U = X\setminus B$ is a nonempty Zariski-open subset of $X$, and the restriction $\Psi_\eta|_U\colon U\to\bB^n\setminus\Psi_\eta(T)$ is finite and flat, so it is open by~\cite[Proposition~3.2.7]{berkovic93:etale_cohomology}.  It follows that $\Psi_{\xi_d}|_{U\cap X_{\xi_d}}\colon U\cap X_{\xi_d}\to\bB^n_{\xi_d}\setminus\Psi_\eta(T)$ is again an open map.%
  \footnote{Let $f\colon X\to Y$ be a continuous, open map of topological spaces, and let $Z\subset Y$ be a subset.  Then $f|_{f\inv(Z)}\colon f\inv(Z)\to Z$ is open with respect to the subspace topologies.}

  Since $\xi_d$ is an Abhyankar point of $\bB^d$, the fiber $X_{\xi_d}$ has pure dimension $n-d$ by~\cite[Lemma~1.5.11]{ducros18:families}.  We claim that $X_{\xi_d}\cap B$ has dimension smaller than $n-d$.  Since $X_{\xi_n}$ is disjoint from $B$, it suffices to show that $X_{\xi_n}$ meets every irreducible component of $X_{\xi_d}$.  Taking the fiber of the triangle~\eqref{eq:Phi.Psi} over $\xi_d\in\bB^d$ yields a finite morphism $\Psi_{\xi_d}\colon X_{\xi_d}\to\bB^n_{\xi_d} = \bB^{n-d}_{\sH(\xi_d)}$ of strictly analytic spaces over $\sH(\xi_d)$.  By~\cite[Proposition~2.10]{gubler_rabinoff_werner:tropical_skeletons}, the canonical reduction $\td X_{\xi_d}$ has pure dimension $n-d$ and the morphism of canonical  reductions $\td\Psi_{\xi_d}\colon\td X_{\xi_d}\to\td\bB^{n-d}_{\sH(\xi_d)} = \bA^{n-d}_{\td k(x_1,\ldots,x_d)}$ is finite.  It follows that $\td\Psi_{\xi_d}$ maps the generic points of $\td X_{\xi_d}$ to the generic point of $\bA^{n-d}_{\td k(x_1,\ldots,x_d)}$, so by~\cite[Proposition~2.4.4]{berkovic90:analytic_geometry}, the map $\Psi_{\xi_d}$ takes the Shilov boundary of $X_{\xi_d}$ to the Shilov boundary $\{\xi_n\}$ of $\bB^{n-d}_{\sH(\xi_d)}$.  Hence the Shilov boundary of $X_{\xi_d}$ is contained in $X_{\xi_n}$.  By~\cite[Proposition~2.15]{gubler_rabinoff_werner:tropical_skeletons}, the Shilov boundary of $X_{\xi_d}$ is equal to the union of the Shilov boundaries of the irreducible components of $X_{\xi_d}$, all of which are nonempty by~\cite[Proposition~2.4.4]{berkovic90:analytic_geometry}.  This proves the claim.

Since $X_{\xi_d}$ is pure dimensional and $X_{\xi_d}\cap B$ has smaller dimension, the set $U\cap X_{\xi_d}$ is dense in $X_{\xi_d}$, so $U\cap\Omega\neq\emptyset$.  Since $\Psi_{\xi_d}|_{U\cap X_{\xi_d}}\colon U\cap X_{\xi_d}\to\bB^n_{\xi_d}\setminus\Psi_\eta(T)$ is an open map, the subset $V\coloneq\Psi_{\xi_d}(U\cap\Omega)$ is nonempty and open in  $\bB^n_{\xi_d}=\bB^{n-d}_{\sH(\xi_d)}$.  The coordinate map
\[ c\colon \bar{\sH(\xi_d)}{}^{n-d} \To \bB^{n-d}_{\sH(\xi_d)} \]
is continuous with respect to the analytic topology on $\bB^{n-d}_{\sH(\xi_d)}$ and the usual (ultrametric) topology on $\bar{\sH(\xi_d)}{}^{n-d}$, and the set of rig-points of $\bB^{n-d}_{\sH(\xi_d)}$ is dense by~\cite[Proposition~2.1.15]{berkovic90:analytic_geometry}, so $c\inv(V)$ is a nonempty open subset of $\bar{\sH(\xi_d)}{}^{n-d}$.  The algebraic closure $\bar Q$ is dense in $\bar{\sH(\xi_d)}$, as the completion of $\bar Q$ is an algebraically closed field  \cite[Proposition 3.4.1/3]{bosch_guntzer_remmert84:non_archimed_analysis} containing $\sH(\xi_d) = \hat Q$ where the latter denotes the completion of $Q$.  Thus there exists $(\alpha_{d+1},\ldots,\alpha_n)\in\bar Q{}^{n-d}$ mapping to a point $y$ in $V\subset \Psi_{\xi_d}(\Omega)$. 
\end{proof}

We will apply Lemma~\ref{lem:coordinate.density} to $\Omega = \red\inv(\eta_W) = X_{\xi_d}\cap\red\inv(W)$ which is a non-empty open subset of $X_{\xi_d}$ by \artref{sec:fibers}.  Recall that we are assuming $\sA$ to be an integral domain.

\begin{lem}\label{lem:point.not.dense}
  Let $x\in X_{\xi_d}$, and suppose that $y = \Psi_{\xi_d}(x)$ has coordinates in $\bar Q$.  Then $\{x\}$ is not Zariski-dense in $X$.
\end{lem}

\begin{proof}
  Suppose that $y$ has coordinates $(\alpha_{d+1},\ldots,\alpha_n)\in\bar Q{}^{n-d}$.   Choose a nonzero polynomial $f\in Q[x_{d+1},\ldots,x_n]$ such that $f(\alpha_{d+1},\ldots,\alpha_n) = 0$.  Since $f$ is a \emph{polynomial} in $x_{d+1},\ldots,x_n$, we may clear denominators to assume $f\in k\angles{x_1,\ldots,x_d}[x_{d+1},\ldots,x_n]$.  Thus $f$ defines a nonzero analytic  function on $\bB^n$ that vanishes on $y$.  Then $g = \Psi_\eta^*(f)$ is a nonzero analytic function on $X$ that vanishes on $x$.
\end{proof}

\begin{proof}[Proof of Theorem~\ref{thm:affine lifting}]
  By Lemma~\ref{lem:coordinate.density}, there exists $x\in\red\inv(\eta_W) = X_{\xi_d}\cap\red\inv(W)$ such that $\Psi_{\xi_d}(x)$ has coordinates in $\bar Q$, and by Lemma~\ref{lem:point.not.dense}, the Zariski-closure $Y$ of $\{x\}$ in $X$ has dimension strictly smaller than $n$.  Since $\red(x)=\eta_W$, we have $W=\overline{\red(x)}\subset  (\bar Y)^\fX$.  Replacing $\fX$ by $(\bar Y)^\fX$, we win by induction on $n$.
\end{proof}

\section{The local lifting theorem} \label{section:local lifting theorem}

We generalize the affine lifting theorem~\ref{thm:affine lifting} to the case of an arbitrary formal $k^\circ$-model $\fX$ of a good strictly $k$-analytic space $X$.  In this situation, we prove only a local result.  We thank Antoine Ducros warmly for helping us with the proof.

\begin{thm} \label{local lifting theorem}
  Let $X$ be a good, compact, strictly $k$-analytic space, let $\fX$ be an admissible formal $k^\circ$-model of $X$, let $x\in X$, and let $\eta\in\fX_s$ be the reduction of $x$.  Suppose that the Zariski closure of $\{\eta\}$ has dimension $d$.  Then for every neighborhood $W$ of $x$, there is a point $z\in W\cap\red\inv(\eta)$, a neighborhood $U$ of $z$ contained in $W$, and an irreducible $d$-dimensional Zariski-closed subspace $Y\subset U$  containing $z$.
\end{thm}

\begin{rem} \label{lower bound in local lifting theorem}
  In the situation of the theorem, any Zariski-closed subspace $Y\subset U$ with $z \in Y$ has dimension bounded below by $d$.  This can be seen as follows.  
For a point $z$ of a $k$-analytic space $Z$ we set
\[ d_k(z) = \trdeg(\td\sH(z)/\td k) + \dim_\Q\bigl((|\sH(z)|/|k|)\tensor_\Z\Q\bigr). \] 
It is a fact~\cite[1.4.6, A.4.11]{ducros18:families} that $\dim(Z) = \sup_{z\in Z}d_k(z)$.  Taking $Z = Y$ and $z\in Y$ as in the theorem, we have a canonical inclusion $\kappa(\eta)\inject\td\sH(z)$, where $\kappa(\eta)$ is the residue field of $\eta$.  In particular,
\[ \dim(Y) \geq d_k(z) \geq \trdeg(\td\sH(z)/\td k) \geq \trdeg(\td\kappa(\eta)/\td k) = d. \]
\end{rem}

We will use the following lemma in the  proof.

\begin{lem}\label{lem:lift.extension.ktilde}
  Let $F/\td k$ be a finitely generated field extension of transcendence degree $d$.  There exists a finitely generated field extension $L/k$ of transcendence degree $d$ equipped with an absolute value extending the given absolute value on $k$ and a $\td k$-isomorphism $\td L\isom F$.
\end{lem}

\begin{proof}
  Let $K$ be the field of rational functions $k(x_1,\ldots,x_d)$ endowed with the Gauss absolute value.  The residue field of $K$ is equal to $\td k(x_1,\ldots,x_d)$.  Choosing a transcendence basis for $F/\td k$, we obtain a $\td k$-homomorphism $\td K\to F$.  Since $F/\td k$ is finitely generated of transcendence degree $d$, it follows that $[F:\td K]<\infty$.  Choose an algebraic closure $K^a$ of $K$ and an absolute value on $K^a$ extending the Gauss absolute value on $K$.  Noting that $\td{K^a}$ is algebraically closed, we choose a $\td K$-homomorphism $F\to\td{K^a}$.  Replacing $F$ with its image in $\td{K^a}$, we are reduced to the following statement:

  \begin{itemize}
  \item[(*)] Let $K$ be a field, let $K^a$ be an algebraic closure endowed with a non-Archimedean absolute value, and let $F/\td K$ be a finite subextension of $\td K^a/\td K$.  Then there exists a finite subextension $L/K$ of $K^a/K$ with $\td L = F$.
  \end{itemize}

  Arguing by induction on $m = [F:\td K]$, we may assume that $F = \td K(\td\beta)$ for some primitive element $\td\beta\in F$.  Let $\td f$ be the minimal polynomial of $\td\beta$ over $\td K$.  Choose a monic polynomial $f\in K^\circ[x]$ lifting $\td f$ modulo $K^{\circ\circ}$.  Let $\beta_1,\ldots,\beta_m$ be the roots of $f$ in $K^a$ (counted with multiplicity), so that $f = (x - \beta_1)\cdots(x - \beta_m)\in K^a[x]$.  By Gauss' Lemma, each $\beta_i$ is contained in $(K^a)^\circ$, so $\td f = (x -\td\beta_1)\cdots(x-\td\beta_m)$ in $\td{K^a}[x]$.  In particular, we have $\td\beta = \td\beta_i$ for some $i$.  Set $\beta =\beta_i$ and $L = K(\beta)$, so that $F\subset\td L$.  Again by Gauss' Lemma, the polynomial $f$ is irreducible over $K$, so $m = [L:K]$.  Then $m = [F:\td K]\leq[\td L:\td K]\leq [L:K] = m$, so $\td L = F$, as desired.
\end{proof}

\begin{proof}[Proof of Theorem~\ref{local lifting theorem}]
  First we reduce to the case $W=X$, which we call the global version of Theorem~\ref{local lifting theorem}. By shrinking $W$, we may assume that $W$ is a  strictly affinoid neighborhood of $x$. By Raynaud's theorem \cite[Theorem~8.4.3]{bosch14:lectures_formal_rigid_geometry} and \cite[Corollary~5.4(b)]{bosch_lutkeboh93:formal_rigid_geometry_II}, there is an admissible formal blowing up $\fX' \to \fX$ and a formal open subset $\fW' \subset \fX'$ with $\fW'_\eta = W$. We get an induced morphism $\fW_s' \to \fX_s$ of schemes of finite type over $\td k$. The fiber over $\eta$ is a scheme of finite type over the residue field $\kappa(\eta)$.  This fiber contains $\red_{\fX'}(x)$, so it is non-empty; thus it contains a closed point $\eta'$. We have 
  \[\trdeg(\kappa(\eta')/\td k)=\trdeg(\kappa(\eta')/\kappa(\eta))+\trdeg(\kappa(\eta)/\td k)=0+d=d.\]
From the global version of Theorem~\ref{local lifting theorem} as applied to the formal model $\fW'$ of $W$ and to the point $\eta' \in \fW_s'$, we deduce that there exists $z \in W$ with $\red_{\fX'}(z)=\eta'$, a neighborhood $U$ of $z$ in $W$, and an irreducible $d$-dimensional Zariski-closed subspace $Y\subset U$ containing $z$. Since $\fX'\to \fX$ maps $\red_{\fX'}(z)$ to $\red_\fX(z)$, we conclude that $\red_\fX(z)=\eta$. This proves that Theorem~\ref{local lifting theorem} follows from its global version.

We assume from now on that $W=X$. The residue field $\kappa(\eta)$ is a finitely generated field extension of $\td k$ of transcendence degree $d$.  By Lemma~\ref{lem:lift.extension.ktilde}, there is a finitely generated field extension $L/k$ of transcendence degree $d$ endowed with an absolute value extending the given one on $k$ and a $\td k$-isomorphism  $\rho \colon \td L\isom\kappa(\eta)$.  Choose an integral affine  $k$-scheme $T$ of dimension $d$ with function field $L$.   Since the points of $T^\an$ are those multiplicative seminorms on $\sO(T^\an)$ which extend the given absolute value on $k$, the absolute value on $L$ defines a point  $t \in T^\an$.  Note that $\sH(t)$ is the completion of $L$ and that $\td\sH(t) = \td L$.

\medskip
\noindent\textbf{Reductions of germs.} Let $z$ be a point of the tube $\red\inv(\eta)$.  Since $\red(z) = \eta$, there is a canonical homomorphism $\kappa(\eta) \inject \td{\sH}(z)$.  Let $\P_{\td\sH(z)/\td k}$ be the Riemann--Zariski space of the residue field extension $\td\sH(z)/\td k$.%
\footnote{For a summary about Riemann--Zariski spaces, we refer to~\cite[6.1]{chambert_ducros12:forms_courants}.}  The closure $E\subset\fX_s$ of $\eta$ is a $\td k$-variety, and there is a canonical morphism $\td k(E) = \kappa(\eta)\inject\td\sH(z)$ over $\td k$, where $\td k(E)$ is the function field of $E$. In such a situation, we call $E$ a \emph{premodel} of $\td\sH(z)$ over $\td k$.%
  \footnote{For a \emph{model}, we would require that the embedding $\td k(E)\inject\td\sH(z)$ be an isomorphism.}
  Let $\P_{\td\sH(z)/\td k}\{E\}$ be the set of valuations in $\P_{\td\sH(z)/\td k}$ with center in $E$.  This is a quasi-compact open subset of $\P_{\td\sH(z)/\td k}$ depending only on the germ $(X,z)$; it is called \emph{Temkin's reduction of the germ} and is denoted $\red(X,z)$.  
	
  The identification $\td\sH(t) = \td L$ composed with the isomorphism $\rho\colon\td L\isom\kappa(\eta)$ and the canonical inclusion $\kappa(\eta)\inject\td\sH(z)$ yields a homomorphism $\td\sH(t)\inject\td\sH(z)$.  Restriction of valuations induces a morphism of Riemann--Zariski spaces $\pi\colon\P_{\td\sH(z)/\td k}\to\P_{\td\sH(t)/\td k}$.   The isomorphism $\rho$ also induces an isomorphism $\td\sH(t)\isom\td k(E)$, so we may regard $E$ as a (pre)model of $\td\sH(t)/\td k$, which defines a quasi-compact open subset $\P_{\td\sH(t)/\td k}\{E\}\subset \P_{\td\sH(t)/\td k}$.  By construction and the fact that a valuation always extends along field extensions, we have $\pi\inv(\bP_{\td\sH(t)/\td k}\{E\}) = \P_{\td\sH(z)/\td k}\{E\}$.  Note that this is true for any $z\in\red\inv(\eta)$.

  The analytification of a scheme of finite type over $k$ is boundaryless \cite[Theorem 3.4.1]{berkovic90:analytic_geometry}, so by~\cite[Theorem~4.1]{temkin00:local_properties}, we have $\red(T,t) = \P_{\td\sH(t)/\td k}$.  Temkin has shown~\cite[Theorem~2.4]{temkin00:local_properties} that reduction of germs establishes a bijective correspondence between strictly analytic domains of the germ $(T,t)$ and non-empty quasi-compact open subsets of $\red(T,t) = \P_{\td\sH(t)/\td k}$.   Let $(T_0,t)$ be the strictly analytic domain with germ equal to $\P_{\td\sH(t)/\td k}\{E\}$.  We fix a strictly affinoid domain $T_0\subset T$ representing this germ.  The previous paragraph gives
  \begin{equation}\label{eq:pi.inv.reduction}
    \pi\inv(\red(T_0,t)) = \red(X,z).
  \end{equation}
        
  The special fiber of the $\sH(t)^\circ$-formal scheme $\fX'=\fX \hat\tensor_{k^\circ} \sH(t)^\circ$ is equal to $\fX_s \tensor_{\td k} \td L$, so the isomorphism $\rho\colon\td L\isom\kappa(\eta)$ induces a canonical $\td L$-rational point $\eta'\in\fX_s'$ lying over~$\eta$.  Since the reduction map is an anti-continuous surjective map~\cite[Proposition~2.17]{gubler_rabinoff_werner:tropical_skeletons} and $\{\eta'\}$ is closed, the tube $\Omega = \red\inv(\eta')$ is a non-empty open subset of the $\sH(t)$-analytic space $\fX'_\eta = X\hat\tensor_k\sH(t)$.  Let $p_1\colon X\times_k T_0\to X$ and $p_2\colon X\times_k T_0\to T_0$ be the projection morphisms.  We may identify $X\hat\tensor_k\sH(t)$ with the fiber~$p_2\inv(t)$.   

        \medskip\noindent\textbf{Boundarylessness.} \textit{Claim: Both $p_1$ and $p_2$ are boundaryless at any point $v\in\Omega$.}

\smallskip
In other words, we claim that $\Omega\subset\Int(X\times_k T_0/T_0)\cap\Int(X\times_k T_0/X)$.  Fix $v\in\Omega$.  By definition of the tube $\Omega$, the reduction of $v$ in $\fX'_s$ is equal to~$\eta'$, and $p_2(v) = t$.  Let $z = p_1(v)\in X$, and let $\pi_1\colon\bP_{\td\sH(v)/\td k}\to\bP_{\td\sH(z)/\td k}$ (resp.\ $\pi_2\colon\bP_{\td\sH(v)/\td k}\to\bP_{\td\sH(t)/\td k}$) be the map of Riemann--Zariski spaces corresponding to the $\td k$-algebra homomorphism $p_1^*\colon\td\sH(z)\to\td\sH(v)$ (resp.\ $p_2^*\colon\td\sH(t)\to\td\sH(v)$).  By a result of Temkin~\cite[Theorem~4.1]{temkin00:local_properties}, showing that $p_1$ and $p_2$ are boundaryless at $v$ is equivalent to proving that
\begin{equation} \label{identity of the reduction of germs}
  \pi_1^{-1}(\red(X,z))=\red(X\times_k T_0,v)=\pi_2\inv(\red(T_0,t)).
\end{equation}
By~\cite[Proposition~4.6]{temkin04:local_properties_II}, which is easily adapted to non-graded reductions of germs in the case of good strictly analytic spaces, we have 
\begin{equation} \label{eq: germs in products}
	\red(X \times_k T_0,v) = \pi_1^{-1}(\red(X,z)) \cap \pi_2^{-1}(\red(T_0,t))
\end{equation}
inside of $\bP_{\td\sH(v)/\td k}$, so~\eqref{identity of the reduction of germs} amounts to the equality $\pi_1\inv(\red(X,z))=\pi_2\inv(\red(T_0,t))$.

By functoriality of the reduction map as applied to $p_1\colon\fX\hat\tensor_{k^\circ}\sH(t)^\circ\to\fX$, we have $\red(z) = \red(p_1(v)) = p_1(\red(v)) = p_1(\eta') = \eta$.  Thus there are canonical homomorphisms $\kappa(\eta) \to\td\sH(z)$ and $\kappa(\eta')\to\td\sH(v)$ making the right square commute in the following diagram of field extensions:
\begin{equation}\label{eq:residue.field.diagram}
  \begin{tikzcd}[row sep=tiny]
    & \td\sH(v) \\
    & & \td\sH(z) \ular["p_1^*"'] \\
    & \kappa(\eta') \arrow[uu] \\
    \td\sH(t) \arrow[uuur, bend left, "p_2^*"] \arrow[ur, "\sigma"] \arrow[rr, "\rho"', "\sim"]
    & & \kappa(\eta) \arrow[uu] \arrow[ul, "\iota"']
  \end{tikzcd}
\end{equation}
Here, the map $\sigma \colon\td\sH(t) \to \kappa(\eta')$ is induced by the fact that $\eta'$ is a point of the $\td\sH(t)$-scheme $\fX_s'$, and $\iota \colon \kappa(\eta) \to \kappa(\eta')$ is induced by the fact that $\eta'$ is lying over $\eta$. 
By \eqref{eq:pi.inv.reduction}, we have that $\pi\inv(\red(T_0,t)) = \red(X,z)$.
Recall that $\pi\colon\bP_{\td\sH(z)/\td k}\to\bP_{\td\sH(t)/\td k}$ is induced by the composition of the arrow on the bottom with the arrow on the right. The map $\pi_2\colon\bP_{\td\sH(v)/\td k}\to\bP_{\td\sH(t)/\td k}$ is induced by the extension on the left, and $\pi_1\colon\bP_{\td\sH(v)/\td k}\to\bP_{\td\sH(z)/\td k}$ is induced by the extension on the top right.  If the outer part of the diagram is commutative then we have $\pi_2 = \pi\circ\pi_1$, such that \eqref{eq:pi.inv.reduction} yields $\pi_2\inv(\red(T_0,t))=\pi_1\inv(\red(X,z))$.  

The homomorphism $\kappa(\eta')\to\td\sH(v)$ comes from the reduction map $\red\colon\fX'_\eta\to\fX'_s$ associated to the formal $\td\sH(t)^\circ$-scheme $\fX'$, so it is a map of $\td\sH(t)$-algebras.  This shows that the upper left triangle is commutative.  It remains to see that the bottom triangle is commutative. Recall that $\eta'$ is the image of the map  $\Spec(\td\sH(t))\to\fX_s\tensor_{\td k}\td\sH(t)$ defined by $r\tensor 1$, where $r\colon\Spec(\td\sH(t))\to\fX_s$ is the composition of the map $(\rho\inv)^*\colon\Spec(\td\sH(t)) \isom \Spec(\kappa(\eta))$ with the canonical map $\Spec(\kappa(\eta))\to\fX_s$.  The morphism $r \otimes 1$ induces a $\td\sH(t)$-homomorphism $\varphi\colon \kappa(\eta')\to \td\sH(t)$. As $\sigma$ is induced by the structure map of $\fX_s'= \fX_s \otimes_{\td k} \td\sH(t)$, we have $\varphi \circ \sigma = 1$, so $\varphi$ and $\sigma$ are inverse isomorphisms. The definition of $r \otimes 1$ shows that the triangle 
\begin{equation} \label{eq: residue field triangle}
  \begin{tikzcd}[sep=small]
    & \kappa(\eta') \dlar["\phi"'] \\
    \td\sH(t) & & \kappa(\eta) \ular["\iota"'] \arrow[ll, "\rho\inv"]
  \end{tikzcd}
\end{equation}
commutes. We conclude that all maps in the triangle \eqref{eq: residue field triangle} are isomorphisms. The commutativity of the bottom triangle of \eqref{eq:residue.field.diagram} now follows readily from $\varphi \circ \iota = \rho^{-1}$ and from $\varphi^{-1}=\sigma$. This concludes the proof that $p_1$ and $p_2$ are boundaryless at $v$.

\medskip\noindent\textbf{Flat, quasi-finite multisections.}
Since fibers in the category of analytic spaces have the induced topology \cite[1.4]{berkovic93:etale_cohomology}, there is an open subset $V\subset X \times_k T_0$ with $V \cap p_2\inv(t) = \Omega $. Since $X$ is flat over $k$, it is clear that $X\times_k T_0$ is flat over $T_0$ (in the sense of~\cite{ducros18:families}), so $V\to T_0$ is again flat.  By a theorem of Ducros on multisections~\cite[Theorem~9.1.2(B)]{ducros18:families} (with $Y = V,\;X = T_0,\;\sF=\sO_V,\;y=$ any point of $\Omega$, and $x=t$; here we use that $p_2$ is boundaryless along $\Omega$), there is a strictly $k$-affinoid space $S$ and a commutative triangle
\[
  \begin{tikzcd}
    & V \dar \\
    S \urar["\psi"] \rar["\phi"'] & T_0
  \end{tikzcd}
\]
such that $\phi$ is a finite at a point $s\in\phi\inv(t)$.  Shrinking $S$ and $T_0$ if necessary, we may and do assume that $\phi$ is finite~\cite[Proposition~3.1.4(b)]{berkovic93:etale_cohomology}.  In particular, we have $\dim(S)\leq\dim(T_0)=d$.  Since $\psi(s)$ lies over $t\in T_0$, we have $\psi(s)\in V\cap p_2\inv(t)=\Omega$.  Set $v = \psi(s)$ and $z = p_1(v)\in X$.  As above, the reduction of $z$ is equal to $\eta$, so
\[ d_k(z) \geq \trdeg(\td\sH(z)/\td k) \geq \trdeg(\kappa(\eta)/\td k) = d, \]
where the second inequality holds because $\kappa(\eta)\subset\td\sH(z)$.  Since $z=p_1(v)=(p_1 \circ \psi)(s)$, we have $d_k(s)\geq d_k(z)\geq d$.  Since $\dim(S)=\sup_{s'\in S}d_k(s')\leq d$, it follows that $\dim(S) = d_k(s) = d$.

Composing $p_1$ with $\psi\colon S\to V\subset X\times_k T_0$ gives a map $p_1\circ\psi\colon S\to X$.  We claim that $S_z \coloneq(p_1\circ\psi)\inv(z)$ has dimension zero.  If not, then $\dim(S_z)\geq 1$, so there exists $s'\in S_z$ with $d_{\sH(z)}(s')\geq 1$.  But then
\[ d_k(s') = d_{\sH(z)}(s') + d_k(z) \geq 1 + d_k(z) \geq 1 + d \]
(the first equality is additivity in towers), which contradicts $\sup_{s'\in S}d_k(s') = \dim(S) = d$.

We claim now that $p_1\circ\psi$ is finite at $s$.  Since $s$ is isolated in its fiber,  by~\cite[Proposition~3.1.4(c)]{berkovic93:etale_cohomology} it is enough to show that $p_1\circ\psi$ is boundaryless at $s$.  Since $\phi\colon S\to T_0$ is boundaryless (it is finite) and since $\psi\colon S\to V$ factors $\phi$, it follows from~\cite[Proposition~3.1.3(ii)]{berkovic90:analytic_geometry} that $\psi$ is boundaryless.  We showed before that $p_1$ is boundaryless at $v$, and the composition of two boundaryless morphisms is again boundaryless, so $p_1\circ\psi$ is boundaryless, hence finite, at $s$.

Again by~\cite[Proposition~3.1.4(b)]{berkovic93:etale_cohomology}, there exist affinoid neighborhoods $U'$ of $s$ and $U$ of $z$ such that $p_1\circ\psi$ induces a finite morphism $U'\to U$.  The image $p_1\circ\psi(U')$ is therefore a Zariski-closed $d$-dimensional subspace $Y$ of $U$ containing $z$.  Replacing $Y$ be an irreducible component containing $z$ (see~\cite[\S 1.5]{ducros18:families}), we get the theorem.  Here we use that the dimension of a Zariski-closed subspace containing $z$ is bounded below by $d$, as noted in Remark \ref{lower bound in local lifting theorem}.
\end{proof}

Let $X$ be an analytic space and let $x\in X$.  The infimum of the integers $\dim_k\overline{\{x\}}{}^{V_{\Zar}}$ for $V$ running through the set of analytic \emph{neighborhoods} of $x$ in $X$ is called the \defi{central dimension of $x$ in $X$} and is denoted $\centdim(X,x)$~\cite[Definition~3.2.2]{ducros18:families}.  For any analytic neighborhood $V$ of $x$ in $X$ one has $\dim_k\overline{\{x\}}{}^{V_{\Zar}}\geq d_k(x)$, as noted
in Remark \ref{lower bound in local lifting theorem};  
it follows that $\centdim(X,x)\geq d_k(x)$.  

A point has central dimension zero if and only if it is a rig-point.  The following Corollary thus generalizes the fact that the rig-points are dense in an analytic space.

\begin{cor}\label{cor:cent.dim.dense}
  Let $X$ be a good, strictly $k$-analytic space of pure dimension $n$.  Then for any nonnegative integer $d\leq n$, the set
  \[ \bigl\{ x\in X~:~\centdim(X,x)=d_k(x) = d \bigr\} \]
  is dense in~$X$.
\end{cor}

\begin{proof}
  Let $x\in X$, let $V$ be a strictly affinoid neighborhood of $x$, and let $W$ be the interior of $V$ in $X$; this is an open neighborhood of $x$ in $X$.  Choose an admissible formal model $\fV$ of $V$.  The special fiber $\fV_s$ has dimension $n$.  Let $E\subset\fV_s$ be an irreducible $d$-dimensional closed subset and let $\eta$ be its generic point. We apply Theorem~\ref{local lifting theorem}  to $\fV$ and $W$ to get a neighborhood $U$ of $z$ in $W$, a point $z \in W$ with reduction $\eta$, and a $d$-dimensional Zariski-closed subspace $Y \subset U$ with $z \in Y$.  It follows that $\centdim(X,z)\leq d$, and equality holds by Remark~\ref{lower bound in local lifting theorem}.
\end{proof}

\section{Semipositive $\R$-PL functions} \label{section:characterization of harmonic functions}

Let $k$ be a non-trivially valued non-Archimedean field.  We consider a good, strictly $k$-analytic space $X$.  We refer to \artref{subsection: PL functions} for the definition of PL functions.  
Recall from the introduction that an $\R$-PL function $u\colon X \to \R$ is called \emph{semipositive} at $x \in X$ if the associated residue line bundle $L_u(x)\in\Pic(\red(X,x))_\R$ is nef, where we refer to  \cite[6.16--6.19]{gubler_rabinoff_jell:harmonic_trop} for residue line bundles. Equivalently, there is a strictly affinoid neighbourhood $U$ of $x$ such that $u=h_\fL$ for a formal $k^\circ$-model $(\mathfrak U,\fL)$ of $(U,\sO_U)$ with $\fL$ nef  \cite[Proposition 7.1]{gubler_rabinoff_jell:harmonic_trop}.

We will show first that for an $\R$-PL function on a smooth curve, semipositivity is the same as subharmonicity in the sense of Thuillier's thesis~\cite[D\'efinition 3.1.5]{thuillier05:thesis}. Recall that Thuillier defines a subharmonic function  on a smooth $k$-analytic curve $X$ as a semicontinuous function $u\colon X \to \R \cup \{-\infty\}$ which is not identically $-\infty$ on any connected component of $X$ and which has the property that for every harmonic function $h$ on a strictly affinoid domain $U$ of $X$ with $u|_{\partial U} \leq h|_{\partial U}$, we have $u|_U \leq h$.

\begin{prop}\label{prop:semipositive.ample.curves}
  Let $X$ be a smooth, strictly $k$-analytic curve and let $u\colon X\to\R$ be an $\R$-PL function.  Then $u$ is subharmonic at $x \in X$ in the sense of Thuillier if and only if $L_u(x)$ is nef.
\end{prop}

\begin{proof}
  First we treat the case when $k$ is algebraically closed.  As explained in~\artref{subsection: PL functions}, there exists a compact strictly analytic neighborhood $U$ of $x$ in $X$, a formal $k^\circ$-model $\fU$ of $U$, and a formal model $\fL$ of $\sO_U$ such that $u|_U = h_\fL$.  Since $k$ is algebraically closed, we may assume that the special fiber $\fU_s$ is reduced, which implies that $\fU$ is covered by formal affines of the form $\Spf(\sA^\circ)$ for a strictly $k$-affinoid algebra $\sA$~\cite[\S 2.4]{gubler_rabinoff_jell:harmonic_trop}.  It follows from \cite[Proposition 2.4.4]{berkovic90:analytic_geometry} that for every irreducible component $Y$ of $\fU_s$, there is a unique point $\xi_Y\in U$ reducing to the generic point of $Y$.  In the formalism of real valued forms and currents on Berkovich spaces \cite{chambert_ducros12:forms_courants}, it is well known that $\d'\d''[u]|_U$ is the discrete measure on $U$ supported on the points $\xi_Y$ of $U$ for $Y$ an irreducible component of $\fU_s$, with multiplicity at $\xi_Y$ equal to $\deg_\fL(Y)$.  See~\cite[\S6.9]{chambert_ducros12:forms_courants} for a more general statement.  Wanner showed that subharmonicity of $u$ on $U$ is equivalent to the positivity of the current $\d'\d''[u]$: see~\cite[Theorems~1 and~2]{wanner19:subharmonicity} and use the fact that $X$ is locally algebraic \cite{put80:class_group}.  Thus $\fL_s$ is nef if and only if $u$ is subharmonic on $U$.
	Suppose that $u$ is subharmonic at $x$.  Shrinking $U$, we may assume that $u$ is subharmonic on $U$.    This implies that $\fL_s$ is nef, so it follows from~\cite[Proposition~7.1]{gubler_rabinoff_jell:harmonic_trop} that $L_u(x)$ is nef.  Conversely, if $L_u(x)$ is nef, then by~\cite[Proposition~7.1]{gubler_rabinoff_jell:harmonic_trop}, we may shrink $U$ so that $\fL_s$ is nef, which as above implies that $u$ is subharmonic on $U$.

	It remains to show that subharmonicity of $u$ at $x$ and numerical effectivity of $L_u(x)$ can be checked after extending scalars to the completion of an algebraic closure~$K = (\bar k)^\wedge$.  By~\cite[Proposition~7.6(2)]{gubler_rabinoff_jell:harmonic_trop}, one can check that $L_h(x)$ is nef after base change to~$K$.  The same is true for subharmonicity: this follows directly from Thuillier's definition and the fact that the base change of a harmonic function remains harmonic \cite[Proposition 2.3.18]{thuillier05:thesis}.
\end{proof}

Our next goal is a characterisation of semipositive $\R$-PL functions by restriction to subcurves as in the complex case.  We need the following characterization of boundary points:

\begin{art} \label{interior and proper closure}
	Let $\fX$ be an admissible formal $k^\circ$-scheme and let $x\in\fX_\eta$.  Then $x\in\Int(X) = X\setminus\del X$ if and only if the closure of $\red_\fX(x)$ is a proper $\td k$-scheme.  This follows from~\cite[Lemme~6.5.1]{chambert_ducros12:forms_courants} in the separated case and from~\cite[Corollary~A.4]{vilsmeier21:monge_ampere} in general.
\end{art}

\begin{thm} \label{semipositivity and curve restriction}
	Let $X$ be a good, strictly $k$-analytic space, and let $h\colon X \to \R$ be an $\R$-PL function.  For $x \in X$, the following conditions are equivalent.
	\begin{enumerate}
		\item \label{first semipos-condition}
		The function $h$ is semipositive at $x$.
		\item \label{second semipos-condition}
		There is a neighborhood $W$ of $x$ such for every strictly analytic neighborhood $U \subset W$ of $x$ and every Zariski-closed curve $C$ in $U$, the restriction $h|_C$ is semipositive.
	\end{enumerate}	
\end{thm}

\begin{proof}  
	Since semipositivity is an open property that is stable under pullback \cite[Proposition~7.6]{gubler_rabinoff_jell:harmonic_trop}, it is clear that \eqref{first semipos-condition} yields \eqref{second semipos-condition}.

	Assume  that $h$ satisfies \eqref{second semipos-condition}. By shrinking $W$, we may assume that $W$ is a strictly affinoid neighborhood of $x$ and that there is a formal model $\fW$ of $U$ and a model $\fL$ of $\sO_W$ on $\fW$ with coefficients in $\R$ such that $h=h_\fL$ on $W$.  Let $Y\subset\fW_s$ be a  curve which is proper over the residue field $\td k$. To prove semipositivity at $x$, it is enough to show that $\deg_\fL(Y)\geq 0$. 
	
	Choose any point $y\in W$ reducing to the generic point $\eta$ of $Y$.  By Theorem~\ref{local lifting theorem} as applied to $X = W$ and $\fX=\fW$, there exist $z \in W$ with $\red_{\fW}(z) = \eta$, a neighborhood $U$ of $z$ in $W$, and an irreducible $1$-dimensional Zariski-closed subspace $C$ of $U$ containing $z$. By shrinking $U$, we may assume that $U$ is strictly affinoid. By results of Bosch and L\"utkebohmert~\cite[Theorem~8.4.3]{bosch14:lectures_formal_rigid_geometry}~\cite[Corollary~5.4(b)]{bosch_lutkeboh93:formal_rigid_geometry_II}, there is a formal model $\fU$ of $U$ and a morphism $\pi\colon\fU\to\fW$ extending the inclusion $U\inject W$ on generic fibers.  Let $\fC$ be the formal closure of $C$ in $\fU$.  This is an admissible formal model of $C$ whose special fiber $\fC_s$ is a closed $1$-dimensional subscheme of $\fU_s$ containing $\eta'=\red_{\fU}(z)$. Functoriality of reductions shows that $\pi(\eta')=\eta$, so the closure $Y'$ of $\eta'$ is a curve mapping to $Y$. Since the closure of $\eta$ in $\fW_s$ is the proper curve $Y$, we conclude from~\artref{interior and proper closure} that $z \in W \setminus \partial W$.  The boundary $\partial U$ is the union of $\partial W \cap U$ with the topological boundary of $U$ in $W$~\cite[Proposition~3.1.3]{berkovic90:analytic_geometry}, so since $U$ is a neighborhood of $z$ in $W$, we have $z \in U\setminus \partial U$. Applying~\artref{interior and proper closure} to $U$ and its formal model $\fU$ shows that   $Y'$ is proper as well. Therefore, the map $Y'\to Y$ is finite and surjective. Note that $Y'$ is an irreducible component of $\fC_s$. Applying \eqref{second semipos-condition} to $U$ and $C$, we have that $h|_C$ is semipositive, and so we deduce from~\cite[Lemma~7.4]{gubler_rabinoff_jell:harmonic_trop} as applied to $(\pi^*\fL)|_\fC$ that $\deg_{\pi^*\fL}(Y')\geq 0$. The projection formula then shows 
	$\deg_\fL(Y)\geq 0$, as required.
\end{proof}

To apply results from Thuillier's thesis \cite{thuillier05:thesis}, we would like to use only smooth analytic curves. To give a variant of the above result suitable for this purpose, we introduce the following notations:

For a real function $h$ on $X$ and an analytic field extension $F$ of $k$, we denote by $h_F$ the induced function on the base change $X_F = X \hat{\otimes}_k F$ of $X$. We call an $\R$-PL function $h$ \defi{pluriharmonic} if $\pm h$  are both semipositive. A strictly $k$-analytic space $X$ is called \defi{rig-smooth} if the sheaf of K\"ahler differentials, considered on the $\G$-topology, is locally free of rank equal to the local dimension. The space $X$ is \defi{smooth} if and only if $X$ is rig-smooth and boundaryless. See \cite{ducros18:families} for details (Ducros calls rig-smooth spaces quasi-smooth).

\begin{cor} \label{restriction to smooth curves}
	Let $X$ be a good, strictly $k$-analytic space and let  $h\colon X \to \R$ be an $\R$-PL function. Then the following conditions are equivalent:
	\begin{enumerate}
		\item \label{first condition smooth curve}
		The function $h$ is semipositive (resp.~pluriharmonic).
		\item \label{second condition smooth curve}
                  There is an algebraically closed analytic field extension $F$ of $k$ such that for every rig-smooth analytic curve $C$ over $F$ and every morphism $\varphi\colon C \to X_F$, the function $h_F \circ \varphi$ is semipositive (resp.~harmonic) on $C$.
		\item \label{third condition smooth curve}
                  There is an algebraically closed analytic field extension $F$ of $k$ such that for every smooth analytic curve $C$ over $F$ and every morphism $\varphi\colon C \to X_F$, the function $h_F \circ \varphi$ is semipositive (resp.~harmonic) on $C$.
	\end{enumerate}	
\end{cor}

If \eqref{second condition smooth curve} or \eqref{third condition smooth curve} is true for an algebraically closed analytic field extension $F$ of $k$, then it is true for all analytic field extensions. This follows from  Corollary \ref{restriction to smooth curves} and Proposition \cite[Proposition 7.6(2')]{gubler_rabinoff_jell:harmonic_trop} which allows to check semipositivity over any field extension.

\begin{proof}
  It is enough to prove the claim for semipositive functions, as $h$ is pluriharmonic if and only if $h$ and $-h$ are semipositive. 
  Since semipositivity is stable under pullback and extension of scalars \cite[Proposition~7.6]{gubler_rabinoff_jell:harmonic_trop}, clearly \eqref{first condition smooth curve} yields \eqref{second condition smooth curve}.	
	Since a smooth analytic curve is just a boundaryless rig-smooth curve, it is clear that \eqref{second condition smooth curve} yields \eqref{third condition smooth curve}. In fact, these two conditions are equivalent as an $\R$-PL function is automatically semipositive in every boundary point $y$ of an analytic curve $C$. Indeed, it follows from \artref{interior and proper closure}  
	that $\red(C,y)$ is non-proper and hence there is no condition for semipositivity at $y$.
	
	It remains to show that that \eqref{second condition smooth curve} yields \eqref{first condition smooth curve}.	Assume that \eqref{second condition smooth curve} holds. By base change using \cite[Proposition 7.6(2')]{gubler_rabinoff_jell:harmonic_trop}, we may assume that $k=F$.
	Let $x \in X$ and let $U$ be a  strictly analytic neighborhood of $x$. For any Zariski-closed curve $Y$ in $U$, we will show that $h|_Y$ is semipositive. Replacing $Y$ by $Y_{\rm red}$ and using \cite[Proposition~7.6(3')]{gubler_rabinoff_jell:harmonic_trop}, we may assume that $Y$ is reduced. Let $\varphi \colon C \to Y$ be the normalization of $Y$ (see \cite{berkovic90:analytic_geometry} before Proposition 3.1.8).  Then $C$ is a rig-smooth curve because $k$ is perfect (it is algebraically closed). By \eqref{second condition smooth curve}, the function $h \circ \varphi$ is semipositive  on $C$. Since the normalization morphism is finite and surjective, it follows from \cite[Proposition~7.6(1')]{gubler_rabinoff_jell:harmonic_trop} that $h|_Y$ is semipositive. By Theorem \ref{semipositivity and curve restriction}, we conclude that $h$ is semipositive  at $x$. Since $x$ was an arbitrary point of $X$, we get \eqref{first condition smooth curve}.	
\end{proof}

\begin{cor} \label{maximum of semipositive}
  Let $X$ be a good, strictly $k$-analytic space and let $\Lambda$ be either $\Z$ or $\Q$. Let $h_1,h_2 \colon X \to \R$ be semipositive $\Lambda$-PL functions. Then $\max(h_1,h_2)$ is a semipositive $\Lambda$-PL function.
\end{cor}

Corollary~\ref{maximum of semipositive} fails for $\Lambda$ strictly containing $\Q$, but only because we restrict ourselves to the $\G$-topology generated by \emph{strictly} analytic domains.  See~\cite[Example~ 5.3]{gubler_rabinoff_jell:harmonic_trop}.

\begin{proof}
	We first note that $\max(h_1,h_2)$ is a $\Lambda$-PL function~\cite[Remark~5.2]{gubler_rabinoff_jell:harmonic_trop}. 
	Using Corollary \ref{restriction to smooth curves}, it is enough to prove that $\max(h_1,h_2)$ is semipositive 
	in case of a smooth analytic curve $X$.
	Since the maximum of two subharmonic functions on such a curve is subharmonic \cite[Proposition~3.1.8]{thuillier05:thesis}, the claim follows from Proposition \ref{prop:semipositive.ample.curves}.
\end{proof}

\begin{thm} \label{pointwise convergence of semipositive PL}
  Let $X$ be a good, strictly $k$-analytic space and let $(h_j)_{j \in \N}$ be a sequence of semipositive $\R$-PL functions on $X$ converging pointwise to an $\R$-PL function $h$. Then $h$ is semipositive. 
\end{thm}

For $\Q$-PL functions on the analytification of a proper algebraic variety over $K$, this was shown in \cite[Theorem 1.3]{gubler_martin19:zhangs_metrics}. The weaker statement assuming uniform convergence was proven in \cite[Proposition 7.2]{gubler_kuenneman19:positivity}. 

\begin{proof}
	We can check semipositivity after extending scalars~\cite[Proposition~7.6(2')]{gubler_rabinoff_jell:harmonic_trop}, so we may assume that $k$ is algebraically closed. By Corollary \ref{restriction to smooth curves}, we may assume that $X$ is a smooth analytic curve.  
	Fix a point $x\in X$; we must show that $h$ is semipositive at $x$.
	Since $X$ is locally algebraic by a result of van der Put \cite{put80:class_group}, we may assume that $X$ is an open subset of $C^\an$ for a proper algebraic curve $C$ over $k$. 
	Replacing $X$ by the interior of a compact neighborhood of $x$, we may assume that   all $h_j$ and $h$ extend to $\R$-PL functions $\td h_j$ and $\td h$ on $C^\an$  (\cite[Proposition 5.11]{gubler_kuenneman19:positivity}, \cite[Proposition 2.7]{gubler_martin19:zhangs_metrics}). An $\R$-PL function $\td h$ on $C^\an$ is the same as a metric on $\mathscr O_{C^\an}$ associated to a 
	model $\fL$ with $\R$-coefficients 
	\cite[Proposition 2.10]{gubler_martin19:zhangs_metrics}; we denote the associated formal metric by $\|\phantom{a}\|_{\td h}$. This allows us to use the theory of local heights from \cite[\S 3]{gubler07:tropical_varieties}. Passing to a dominant model, we may assume that $\fL$ is a line bundle with $\R$-coefficients on a formal model $\fC$ of $C$ with reduced special fiber, as explained in~\cite[\S 2.4]{gubler_rabinoff_jell:harmonic_trop}.
	Then, by definition,  the Chambert-Loir measure $c_1(\mathscr O_C,\td h)$ associated to 
	$\td h$ (and $\fL$) is the discrete measure on $C^\an$ supported on the points $z\in C^\an$ reducing to generic points of $\fC_s$; we call such points \defi{divisorial points}. The multiplicity at a divisorial point $z$ is $\deg_\fL(Z)$, where $Z$ is the closure of $\red_\fC(z)$ in $\fC_s$.  These measures were introduced by Chambert-Loir in~\cite{chambertloir06:mesures_equidist}; for the current setting with formal metrics, we refer to~\cite[\S 3]{gubler07:tropical_varieties}. 
	
	We want to prove semipositivity of $h$ on the open subset $X$ of $C^\an$.  It is equivalent to show that the Chambert--Loir measure  $c_1(\mathscr O_C,\td h)$ has positive multiplicity in every divisorial point $x \in X$ associated to $\fC$:  see \cite[Lemma~6.8]{gubler_kuenneman19:positivity} or the proof of Proposition \ref{prop:semipositive.ample.curves}.

	Suppose now that $x$ is a divisorial point in $X$ associated to $\fC$.  By \cite[Lemma~2.13]{gubler_martin19:zhangs_metrics}, there is a $\Q$-PL function $\varphi$ on $C^\an$ with $\varphi(x)=1$ and with support contained in an open neighborhood $W$ of $x$ contained in $X$ and not containing any other divisorial points $z \neq x$.  Using that $\Q$-PL functions are stable under minimum and maximum \cite[Proposition 2.12]{gubler_martin19:zhangs_metrics}, we may assume $0 \leq \varphi \leq 1$. 

	Let $m$ be the multiplicity of $c_1(\sO_C, \td h)$ at $x$.   We have
	\[m = \int_{C^\an} \varphi \,c_1(\sO_C, \td h) \]
	by the choice of $\varphi$. The right hand side is the local height of $C$ with respect to two metrized line bundles $(\sO_C,\|\phantom{a}\|_{\phi})$ and $(\mathscr O_C,\|\phantom{a}\|_{\td h})$ using the canonical section $1$ (and the corresponding trivial Cartier divisor) for the line bundles \cite[3.8]{gubler07:tropical_varieties}. Here the metric $\|\phantom{a}\|_{\phi}$ is determined by $-\log \|1\|_\varphi = \varphi$, and similarly for $\td h$. Now symmetry of the local heights \cite[Proposition 3.9]{gubler07:tropical_varieties} based on commutativity of intersection numbers on formal models \cite[Theorem 5.9]{gubler98:local_heights_subvariet} gives the integration by parts formula
	\[ \int_{C^\an} \varphi \,c_1(\mathscr O_C, \td h)= \int_{C^\an} \td h \,c_1(\mathscr O_C, \varphi). \]
	In \cite[\S6]{chambert_ducros12:forms_courants}, a Monge--Amp\`ere measure was introduced for PL-metrics based on a local analytic construction. It was shown in \cite[\S6.9]{chambert_ducros12:forms_courants} and in \cite[Theorem 10.5]{gubler_kunneman:tropical_arakelov} that this Monge--Amp\`ere measure agrees with the corresponding Chambert-Loir measure and hence the latter is local in the Berkovich topology.
	We conclude  that $c_1(\mathscr O_C, \varphi)$ is a finite discrete measure on $C^\an$ with support contained in $W \subset X$, so we get
	\[ m=\int_{W}   \lim_{j \to \infty} h_j\,c_1(\mathscr O_C, \varphi)=\lim_{j \to \infty}\int_{W}   h_j\,c_1(\mathscr O_C, \varphi). \]
	Now we do the same steps backwards with $h_j$ instead of $h$ to deduce
	$$m=\lim_{j \to \infty}\int_{C^\an}   \td h_j\,c_1(\mathscr O_C, \varphi)=\lim_{j \to \infty}\int_{C^\an}   \varphi\,c_1(\sO_C, \td h_j).$$
	Since $\varphi \geq 0$ with support in $X$ and since $c_1(\sO_C,h_j)$ restricts to a \emph{positive} discrete measure $c_1(\sO_X,  h_j)$ on $X$ as $h_j$ is semipositive on $X$, we obtain
	$$m=\lim_{k \to \infty}\int_{X}   \varphi\,c_1(\sO_X,  h_j) \geq 0,$$
	proving the claim.
\end{proof}

\section{The Bieri--Groves theorem for affinoid spaces} \label{section: Bieri--Groves theorem}

The Bieri--Groves theorem \cite{bieri_groves84:geometry_set_character_induced_valuation} shows that for every closed subscheme of pure dimension $d$ in a multiplicative torus $\bT$ of rank $n$ over $k$, the tropicalization is a union of $d$-dimensional $(\Z,\Gamma)$-polyhedra in $\R^n$. This was generalized to closed analytic subvarieties of polytopal subdomains of $\bT$ in \cite{gubler07:tropical_varieties}. Berkovich and Ducros generalized the Bieri--Groves theorem to arbitrary analytic spaces, but the tropicalization then can be of lower dimension. We give here the pure dimensionality of the tropical variety of an affinoid space away from the tropicalization of the boundary. The arguments rely crucially on results of Ducros on the tropicalization of the germ of an analytic space and on tropical skeletons.

Let $k$ be a non-trivially valued non-Archimedean field with valuation $v$ and value group $\Gamma \subset \R$. 
In this section, we consider a compact strictly analytic space $X$ over $k$.  Note that the strictness assumption on $X$ and the non-triviality of $v$ can be always achieved by base change, so this is not really a restriction of generality as tropical varieties are invariant under base change. 

\begin{art} \label{tropicalization map}
	Consider the split torus $\bT=\mathbb G_{\rm m}^n=\Spec k[\chi_1^\pm,\dots, \chi_n^\pm]$ of rank $n$ over $k$ for which we have the \emph{tropicalization map}
	$$\trop\colon \bT^\an \longrightarrow \R^n \, , \quad x \mapsto (|\chi_1(x)|, \dots, |\chi_n(x)|).$$
	The classical Bieri--Groves theorem \cite{bieri_groves84:geometry_set_character_induced_valuation} shows that for any closed subvariety $Y$ of $\bT$ of dimension $d$, the \emph{tropical variety} $\Trop(Y)\coloneqq \trop(Y^\an)$ is a finite union of $(\Z,\Gamma)$-polyhedra of dimension $d$.
\end{art}

\begin{art} \label{smooth tropicalization}
	A morphism $\varphi\colon X \to \bT^\an$ 
	to the split torus $\bT=\mathbb G_{\rm m}^n$ with character lattice $M=\Z^n$ is called a \emph{moment map}. 
	Then 
	$$\varphi_{\trop}\coloneqq \trop \circ \varphi\colon X \longrightarrow \R^n$$
	is called a \emph{smooth tropicalization}. Such maps are at the core in the theory of Chambert--Loir and Ducros \cite{chambert_ducros12:forms_courants} of real valued $(p,q)$-forms on $X$. Results of Berkovich and Ducros show that the tropical variety $\varphi_{\rm trop}(X)$ is a finite union of $(\Z,\Gamma)$-polyhedra of dimension at most $d$: see \cite[Th\'eor\`eme 3.2]{ducros12:squelettes_modeles}. Ducros shows in \textit{loc.\ cit.} that  $\varphi_{\trop}(\partial X)$ is contained in a finite union of $(\Z,\Gamma)$-polytopes in $\R^n$ of dimension at most $d-1$.
\end{art}

\begin{art} \label{graded residue field}
	For $x \in X$, we have the completed residue field $\sH(x)$. We denote by $\td\sH(x)$ the residue field of $\sH(x)$ and by $\td\sH(x)^\bullet$ the graded residue field of $\sH(x)$, in the sense of~\cite{ducros12:squelettes_modeles}. We have a canonical homomorphism $\chi_x\colon \sO_{X,x} \to \sH(x)$; we let $\td\chi_x^\bullet$ be the composition of $\chi_x$ with $\sH(x)\to \td\sH(x)^\bullet$. Let $\td M(x) = \{\td\chi_x^\bullet(\phi^*u)\mid u\in M\}$ and let $\td K^\bullet(\td M(x))$ be the graded subfield of the graded residue field $\td\sH(x)^\bullet$ generated over $\td K^\bullet$ by $\td M(x)$. We set as in \cite[3.6]{gubler_rabinoff_jell:harmonic_trop}
	\[ d_\phi(X,x) \coloneq \trdeg\bigl(\td K^\bullet(\td M(x))/\td K^\bullet\bigr). \]
\end{art}

A basic tool for the study of  the tropical variety $\varphi_{\trop}(X)$ is the following local result of Ducros.
\begin{thm}[{\cite[Th\'eor\`eme 3.4]{ducros12:squelettes_modeles}}] \label{Ducros results about tropical germs}
	For $x \in X$,  there is a compact strictly analytic neighbourhood $V(x)$ of $x$ such that for all compact strictly analytic neighbourhoods $W \subset V(x)$ of $x$, the germs of the tropical varieties $\varphi_{\rm trop}(W)$ and $\varphi_{\rm trop}(V(x))$ agree at $\omega = \varphi_{\rm trop}(x)$.  We denote this germ at $\omega$ by $\varphi_{\rm trop}(X,x)$. The following properties hold:
	\begin{enumerate}
        \item \label{Ducros1} The germ $\phi_{\trop}(X,x)$ is a rational fan of dimension at most $d_\phi(X,x)$.
		\item \label{Ducros2}  If $x \not\in \partial X$, then $\phi_{\trop}(X,x)$ has pure dimension $d_\varphi(X,x)$. 
	\end{enumerate}
\end{thm}

Next, we will recall the tropical skeleton of $\varphi$ from \cite[Section 3]{gubler_rabinoff_jell:harmonic_trop}.

\begin{art} \label{tropical skeleton}
	Assume that $X$ has pure dimension $d$. The \emph{tropical skeleton of $\varphi$} is defined by $S_\varphi(X) \coloneqq \{x \in X \mid d_\varphi(X,x)=d\}$. By results of Ducros \cite{ducros12:squelettes_modeles}, the tropical skeleton is a closed subset of $X$ consisting only of Abhyankar points and has a canonical piecewise linear structure such that $\varphi_{\trop}$ restricts to a   surjective piecewise linear map  $S_\varphi(X) \to \varphi_{\rm trop}(X)$  with finite fibers. 
	We refer to \cite[Section 3]{gubler_rabinoff_jell:harmonic_trop} for details.  
	For $x \in \Int(X)=X \setminus \partial X$, it follows from Theorem \ref{Ducros results about tropical germs} that $x \in S_\varphi(X)$ if and only if $\dim(\varphi_{\rm trop}(X,x))=d$.
\end{art}

We are now ready to state our global pure dimensionality result for affinoid spaces.	

\begin{thm} \label{affinoid Bieri-Groves theorem}
	Let $X=\sM(\sA)$ be a strictly affinoid space over $k$ of pure dimension $d$ and let $\varphi_{\rm trop}\colon X \to \R^n$ be a smooth tropicalization map. Then the tropical variety $\varphi_{\rm trop}(X)$ has pure dimension $d$ at every point of $\varphi_{\rm trop}(X)\setminus \varphi_{\rm trop}(\del X)$.
\end{thm}

Note that one cannot drop the requirement that $X$ be affinoid: indeed, if $X$ is proper then it is compact and has no boundary, but $\phi_{\trop}(X)$ could have dimension zero.

\begin{proof}
	Let $\omega\in\phi_{\trop}(X)\setminus\phi_{\trop}(\del X)$.  If there exists a point $x$ of $S_\phi(X)$ contained in $\phi_{\trop}\inv(\omega)$, then \artref{tropical skeleton} shows $\dim\phi_{\trop}(X,x) = d$ since $x\in\Int(X)$, which implies that the fan of $\phi_{\trop}(X)$ at $\omega$ has dimension $d$; if this is true for all points in the open subset $\phi_{\trop}(X)\setminus\phi_{\trop}(\del X)$ of $\phi_{\trop}(X)$, then it follows that $\phi_{\trop}(X)$ has pure dimension $d$ at all such points.  Therefore it is enough  to show that $S_\phi(X)\cap\phi_{\trop}\inv(\omega)\neq\emptyset$.

	Fix $\omega\in\phi_{\trop}(X)\setminus\phi_{\trop}(\del X)$.  Since the tropical skeleton of a base change with respect to an analytic field extension of $k$ maps onto the original tropical skeleton \cite[Proposition 3.12]{gubler_rabinoff_jell:harmonic_trop}, we may assume that $\omega$ is $\Gamma$-rational. 
	Let $q\colon \bT \to \bT'$ be a homomorphism of tori. Such a homomorphism induces a homomorphism $Q\colon N \to N'$ between the cocharacter lattices $N,N'$ of $\bT$ and $\bT'$, respectively, and vice versa. We take for $\bT$ the split torus of rank $n$ given by the target of the moment map $\varphi\colon X \to \bT^\an$ and for $\bT'=\mathbb G_{\rm m}^d$ a split torus of rank $d=\dim(X)$. Then $Q$ induces a linear map $Q_\R\colon \R^n \to \R^d$. Since $\dim\varphi_{\rm trop}(X) \leq d$, we may choose $Q$ generically, meaning that $Q_\R$ is injective on every face of $\varphi_{\rm trop}(X)$. 
	Since $\varphi_{\rm trop}(\del X)$ is contained in a polyhedral complex of dimension strictly less than $d$, we may additionally assume  that $\omega' = Q_\R(\omega)$ is not contained in $\psi_{\trop}(\del X)$, where $\psi \coloneqq q^\an \circ \varphi\colon X \to \bT'^\an$.  Let $U_{\omega'} = \trop\inv(\omega')\subset\bT'{}^\an$ and let $X_{\omega'} = \psi_{\trop}\inv(\omega') = \psi\inv(U_{\omega'})$. Since $\omega$ is $\Gamma$-rational,  these are strictly affinoid domains in $\bT'^\an$ and $X$, respectively.  We have $\Int(X) = \Int(X/\bT'{}^\an)$ by~\cite[Proposition~3.1.3]{berkovic90:analytic_geometry} since $\del\bT'{}^\an = \emptyset$, and also $\Int(X_{\omega'}/U_{\omega'})\supset\Int(X/\bT'{}^\an)\cap X_{\omega'}$ by \textit{loc.\ cit.}, so $\del(X_{\omega'}/U_{\omega'}) = \emptyset$ since $X_{\omega'}\subset\Int(X)$.  Since $X_{\omega'}$ is compact and separated, it follows that $X_{\omega'}\to U_{\omega'}$ is a proper morphism of strictly affinoid spaces. It follows from the direct image theorem that the morphism $\varphi\colon X_{\omega'}\to U_{\omega'}$ is finite.
	
	Using that $Q_\R$ is injective on every face of $\varphi_{\rm trop}(X)$, it is clear that $Q_\R\inv(\omega')\cap\phi_{\trop}(X)$ is a finite set. It follows that $X_\omega = \phi_{\trop}\inv(\omega)$ is open and closed in the strictly affinoid space $X_{\omega'} = \psi_{\trop}\inv(\omega') = \phi_{\trop}\inv(Q_\R\inv(\omega'))$, so $\psi\colon X_\omega\to U_{\omega'}$ is finite as well.
	Note that the affinoid domain $U_{\omega'}$ has a unique Shilov boundary point $\xi'$ \cite[Corollary 4.5]{gubler07:tropical_varieties}. 
	The Shilov boundary points of $X$ correspond bijectively to the generic points of the canonical reduction $\td X =\Spec(\td\sA)$ of $X$ and the correspondence is given by the reduction map \cite[Proposition 2.4.4]{berkovic90:analytic_geometry}. By functoriality of the reduction map and since the reduction $\td\psi\colon \widetilde{X_{\omega}}\to \widetilde{U_{\omega'}}$ is a finite morphism of affine schemes of pure dimension $d$ \cite[Theorem 6.3.4/2]{bosch_guntzer_remmert84:non_archimed_analysis}, we conclude that every Shilov boundary point $x$ of $X_\omega$ maps  to $\xi'$. Since $\xi'$ is the Gauss point of $\bT'^\an$ of weight $\omega'$, we  deduce that 
	$\trdeg(\td K^\bullet(\td M'(\xi'))/\td K^\bullet)=d$ for the character lattice $M'$ of $\bT'$. Using that $x$ is finite over $\xi'$, we get 
	$$d_\psi(X,x)=d_\psi(X_\omega,x)=
	\trdeg(\td K^\bullet(\td M'(x))/ \td K^\bullet(\td M'(\xi')))+\trdeg(\td K^\bullet(\td M'(\xi'))/\td K^\bullet)=d.$$ 
	We conclude that $x \in S_\psi(X)\cap \varphi_{\rm trop}^{-1}(\omega)$. Since $\td M'(x) \subset \td M(x)$, we have  $S_\psi(X) \subset S_\phi(X)$. Then $x \in S_\phi(X)\cap \varphi_{\rm trop}^{-1}(\omega)$, proving the claim.
\end{proof}

\section{The maximum principle} \label{section: maximum principle}

In this section, we prove the maximum principle for semipositive $\R$-PL functions, and more generally, for piecewise smooth psh functions. We start with a tropical version. We refer to
\artref{subsection: polyhedral geometry} for the notions used from polyhedral geometry and to \cite[Section 9]{gubler_rabinoff_jell:harmonic_trop} for the balancing condition.

\begin{art} \label{tropical notation}
	Let $C$ be a tropical variety of pure dimension $d$ in an open subset $U$ of $\R^n$.  This means that $C$ is a  locally finite $(\Z,\R)$-polyhedral complex $\sC$ in $U$ (up to refinement) with positive constant weights $m$ on the maximal faces  which satisfy the balancing condition.  A function $f\colon C \to \R$ is called \emph{piecewise smooth} with respect to $\sC$ if $f$ is a (continuous) function which is smooth on each face of $\sC$. For such a piecewise smooth function $f$, we can define the \emph{corner locus} $f \cdot C$, which is a tropical cycle of dimension $d-1$ with smooth  (but not necessarily positive) weights in $U$ whose underlying polyhedral complex is as a subcomplex of $\sC$. See \cite[Definition 1.10]{gubler_kunneman:tropical_arakelov} for details. 
\end{art}

\begin{art} \label{tropical psh functions}
  Let $C$ be a tropical variety of pure dimension $d$ and let $f\colon C \to \R$ be a piecewise smooth function as in \artref{tropical notation}.  We will use real valued $(p,q)$-forms on $C$ and the dual notion of Lagerberg currents; for details, see~\cite[\S 2.7]{gubler_rabinoff_jell:harmonic_trop}.  We call $f$ \emph{weakly tropically psh} if $\d'\d''[f]$ is a positive Lagerberg current on $C$. Here, $[f]$ denotes the current associated to $f$ and $\d'\d''[f]$ is in the sense of currents.  We denote the naive differentials in the sense of functions by 
	$\dpa f, \dpb f$ or $\dpa \dpb f$, which give piecewise smooth Lagerberg forms~\cite[Remark 3.11]{gubler_kunneman:tropical_arakelov}. Using the tropical Poincar\'e--Lelong formula \cite[Corollary 3.18]{gubler_kunneman:tropical_arakelov}, we have
	\begin{equation} \label{tropical Poincare-Lelong}
		\d'\d''[f]= [\dpa \dpb f] + \delta_{f\cdot C},
	\end{equation}
	with $	\delta_{f\cdot C}$ the current of integration over the corner locus of $f$.

        We claim that $f$ is weakly tropically psh if and only if $\dpa \dpb f$ is a positive (piecewise smooth) $(1,1)$-form and $f \cdot C$ has nonnegative weight functions.
	The non-trivial direction is to show that if $\d'\d''[f]$ is positive, then both $[\dpa \dpb f]$ and $\delta_{f\cdot C}$ are positive.  To see that, we choose a compactly supported positive $(d-1,d-1)$-form $\eta$ on $C$. Standard arguments using partition of unity give a decreasing sequence of smooth compactly supported functions $\varphi_i$ with $0 \leq \varphi_i \leq 1$ such that $\lim_{i \to \infty} \varphi_i$ is the characteristic function $\chi$ of the support of $f \cdot C$. For every $i \in \N$, we have 
    \begin{equation} \label{first positivity}
        0 \leq \d'\d''[f](\varphi_i\eta)=[\dpa \dpb f](\varphi_i\eta)+\delta_{f\cdot C}(\eta).
    \end{equation}
Similarly, for $\psi_i\coloneqq 1-\varphi_i$, we get 
\begin{equation} \label{second positivity}
	0 \leq \d'\d''[f](\psi_i\eta)=[\dpa \dpb f](\psi_i\eta)
\end{equation}
using that $\psi_i$ vanishes on the support of $f \cdot C$. 
	Note that $\dpa \dpb f \wedge \eta$ is a piecewise smooth form on $C$ which induces a signed measure on $C$, hence the  monotone convergence theorem yields 
	$$\lim_{i \to \infty}[\dpa \dpb f](\varphi_i\eta)= \int_{C} \chi\dpa \dpb f \wedge \eta=0$$
	as the integral of a smooth form over a lower dimensional polyhedral set is zero. By \eqref{first positivity}, we get $\delta_{f\cdot C}(\eta) \geq 0$ which means that $\delta_{f \cdot C}$ is positive. 
	
	Since the $\psi_i$ form an increasing sequence converging to $1-\chi$, the monotone convergence theorem yields similarly
	$$\lim_{i \to \infty}[\dpa \dpb f](\psi_i\eta)=\int_{C} (1-\chi)\dpa \dpb f \wedge \eta=\int_{C} \dpa \dpb f \wedge \eta.$$
	It follows from \eqref{second positivity} that $[\dpa \dpb f](\eta) \geq 0$ which means that $[\dpa \dpb f]$ is positive.
\end{art}

\begin{rem}
  The definition of a weakly tropically psh function is adapted to the corresponding notion of convexity introduced by Botero, Burgos and Sombra in~\cite{botero_burgos_sombra:convex}, at least in the case of piecewise affine functions.  We call them ``weakly'' psh because the authors consider three different notions of concavity for piecewise affine functions on a tropical cycle; our notion corresponds to the weakest, which has the disadvantage that the induced Monge--Amp\`ere measure is not always positive. 
\end{rem}

The following result is the \emph{maximum principle for weakly tropically psh funnctions}.

\begin{prop} \label{tropical maximum principle}
  Let $C$ be a tropical variety of pure dimension $d$ in an open subset $U$ of $\R^n$, and let $f \colon C \to \R$ be a piecewise smooth function that is weakly tropically psh~\artref{tropical psh functions}. If $f$ has a local maximum at $\omega \in C$, then $f$ is constant in a neighbourhood of $\omega$.
\end{prop}

\begin{proof}
	This is a local statement, so we may assume that $\omega=0$ and that $U$ is an open ball with center $0$. Moreover, we may assume that there is a tropical fan $\Sigma$ of pure dimension $d$ in $\R^n$ such that $C=\Sigma \cap U$ and such that for each cone $\sigma$ of $\Sigma$, the function $f|_\sigma$ is smooth. Then we claim  $f$ is constant. Clearly, this yields the proposition.
	
	We prove the claim by induction on $d$. If $d=0$, then the claim is obvious, but we need also to prove the case $d=1$. Then $C$ is a  graph with weights $m_e$ along the edges $e$ satisfying the balancing condition. The corner locus is then supported in $\omega=0$ where it has multiplicity
	$$m = \sum_e m_e\,\frac{\d f}{\d t_e}(0) $$
	where $e$ ranges over all semiopen edges of $C=\Sigma \cap U$, where $t_e$ denotes the parametrization of the edge $e$ and where $\frac{\d f}{\d t_e}(v)$ is the outgoing slope at $0$ along $e$. Since $\omega=0$ is a local maximum and $m_e >0$ for all edges $e$, we conclude that $\frac{\d f}{\d t_e}(0) \leq 0$ for all the outgoing slopes. Since $f$ is weakly tropically psh, it follows from \artref{tropical psh functions} that all these outgoing slopes are $0$. Moreover, we also have seen that $\dpa \dpb f$ is a positive form on $e$, which means that $(\frac{\d}{\d t_e})^2 f  \geq 0$ on $e$, and hence $f$ is increasing on $e$. Since $\omega$ is a local maximum, we conclude that $f$ is constant on every edge $e$, which proves the case $d=1$.
	
	Now assume $d \geq 2$. We assume that $f$ is not constant on $C=U \cap \Sigma$. Then there is $\omega' \in C$ such that $f(\omega')\neq f(\omega)$. By density of the rational points in $|\Sigma|$, we may assume that $\omega' \in \Q^n$. We pick a generic rational linear hyperplane $H$ through $\omega'$ (and $0$). Using tropical intersection theory~\cite{allermann_rau10:tropical_intersection_thy}, we get a $(d-1)$-dimensional effective tropical fan $\Sigma' \coloneqq H \cdot \Sigma$. Using that $H$ is generic and that $d \geq 2$, this is a proper intersection and hence $\Sigma'$ has support equal to $|\Sigma|\cap H$. It follows that
	$\omega'$ is contained in the $(d-1)$-dimensional tropical variety $C' \coloneqq \Sigma' \cap U$ and that $f' \coloneqq f|_{C'}$ is smooth on each face of $C'$. We claim that $f'$ is weakly tropically psh. Assuming this, we apply the induction hypothesis to deduce that $f'$ is constant, which contradicts $f'(\omega)\neq f'(\omega')$. This will prove $f$ is constant on $C =\Sigma \cap U$. 
	
	To see that $f'$ is weakly tropically psh, we first note that $\dpa \dpb f$ is a positive piecewise smooth form on $C$ as we have seen in \artref{tropical psh functions}. We conclude that $\dpa\dpb f'=(\dpa \dpb f)|_{C'}$ is a positive piecewise smooth form on $C'$. On the other hand, it is shown in \cite[Proposition 1.14]{gubler_kunneman:tropical_arakelov} that $f'\cdot C'=f'\cdot (C\cdot H)= (f\cdot C) \cdot H$ holds, and hence the corner locus $f' \cdot C'$ has non-negative smooth weight functions. It follows from \artref{tropical psh functions} that $f'$ is weakly tropically psh. 
\end{proof}

				\begin{art} \label{piecewise smooth forms}
				Let $X$ be a good strictly analytic space over $k$. Recall that Chambert--Loir and Ducros \cite{chambert_ducros12:forms_courants} introduced (real valued) smooth $(p,q)$ forms on $X$ by using pull-backs of smooth Lagerberg forms  with respect to smooth tropicalization maps and a sheafification process. This leads to a bigraded sheaf of smooth forms on $X$ with natural differentials $\d'$ and $\d''$.
				
				Replacing smooth Lagerberg forms by piecewise smooth Lagerberg forms, the same process leads to a bigraded sheaf of piecewise smooth forms on $X$ with differentials $\dpa$ and $\dpb$.  More explicitly, in an affinoid neighborhood $U$ of a point of $x$, a piecewise smooth form is equal to the pullback of a piecewise smooth Lagerberg form $\alpha$ under a smooth tropicalization map $\phi_{\trop}\colon U\to\R^n$.  See \cite[Section 10]{gubler_rabinoff_jell:harmonic_trop} for more details.
				
				The piecewise smooth forms of bidegree $(0,0)$ are special continuous functions which we call \defi{piecewise smooth functions}. It is clear that every $\R$-PL function is a piecewise smooth function where we  choose all the $\omega_i$'s above to be affine functions. 
				\end{art}
				
				\begin{art} \label{currents and psh}
				Chambert-Loir and Ducros introduce currents on a separated good strictly analytic space $X$ as a dual notion to smooth forms \cite[Section 4]{chambert_ducros12:forms_courants}. These also form a bigraded sheaf with natural differentials $\d'$ and $\d''$.  For any continuous function $h\colon X \to \R$, there is an associated current $[h]$ defined by integrating $h$ against compactly supported smooth forms, as explained in~\cite[\S 5.4]{chambert_ducros12:forms_courants}. The function $h$ is called \emph{plurisubharmonic} (or \defi{psh} for short) if $\d'\d''[h]$ is a positive current.  See \cite[\S 5.5]{chambert_ducros12:forms_courants} for details.
				
				Note that plurisubharmonicity is a local property, so it is defined even without assuming $X$ to be separated.

                                By~\cite[Theorem~7.14]{gubler_rabinoff_jell:harmonic_trop}, a semipositive $\R$-PL function is a plurisubharmonic piecewise smooth function in this sense.
				\end{art}
				
				We  prove now the \emph{maximum principle for piecewise smooth psh functions}.
				
				\begin{thm} \label{thm: maximum principle for semipositive functions}
                                  Let $X$ be a good strictly analytic space over $K$ and let $h$ be a piecewise smooth psh function on $X$. If $h$ has a local maximum at $x \in X \setminus \partial X$, then $h$ is locally constant at $x$.  
				\end{thm}
				
				\begin{proof}
					First we replace $X$ by an irreducible component containing $x$ to assume that $X$ has pure dimension $d$.
					By \artref{piecewise smooth forms},  there is a strictly affinoid neighbourhood $U$ of $x$,  a smooth tropicalization map $h'\colon U \to \R^n$, and a piecewise smooth function $F \colon h'(U) \to \R$ such that $h|_U=F \circ h'$ on $U$.   
                                        Since the boundary of $U$ is the boundary of $X$ together with the topological boundary of $U$ in $X$ \cite[Proposition 3.1.3]{berkovic90:analytic_geometry}, we conclude $x \not \in \partial U$ and so we may assume $U=X$. In particular, $X$ is strictly affinoid. We may also assume that $h$ has a global maximum at $x$.

					We claim that there exists a strictly affinoid neighborhood $U$ of $x$ and a smooth tropicalization $h_0\colon U\to\R^m$ such that $h_0(x)\notin h_0(\del U)$.  By~\cite[Lemme~3.4.3]{chambert_ducros12:forms_courants}, there exist analytic functions $f_1,\ldots,f_m$ on $X$ such that:
					\begin{enumerate}
						\item $|f_i|\leq 1$ on $X$ for each $i$,
						\item $|f_i(x)|<1$ for each $i$, and
						\item for every $y\in\del X$, there is some $i$ with $|f_i(y)|=1$.
					\end{enumerate}
					(Here we can take all spectral radii equal to $1$ for simplicity since $X$ is strictly affinoid.)  If $f_i(x) = 0$ for some $i$, then we may replace $f_i$ by $f_i+\lambda$ for any $\lambda\in K^\times$ with $|\lambda| < 1$ without affecting properties (1)--(3) above; thus we assume $f_i(x)\neq 0$ for all $i$.  Choose $r\in\log\sqrt{|K^\times|}$ such that $-\log|f_i(x)| < r$ for all $i$, let $U\subset X$ be the Laurent domain $\bigcap_{i=1}^m\{-\log|f_i|\leq r\}$, and let $h_0 = (f_1,\ldots,f_m)_{\trop}\colon U\to\R^m$.  Then $h_0(U\cap\del X)$ is contained in the union of the coordinate hyperplanes, and $h_0\inv((-\infty,r)^m) = \bigcap_{i=1}^m\{-\log|f_i|<r\}$ is contained in the topological interior of $U$ in $X$ since it is open in $X$. Using that $\del U$ is equal to $U\cap\del X$ union the topological boundary of $U$ in $X$ \cite[Proposition 3.1.3]{berkovic90:analytic_geometry}, it follows that $h_0(\del U)$ is contained in the boundary of the box $[0,r]^m$.  Since $h_0(x)\in(0,r)^m$, the claim follows.

					Replacing $X$ by $U$ and $h'$ by the product of $h'$ and $h_0$, we may assume both that $h = F\circ h'$ for a piecewise  smooth function $F\colon h'(X)\to\R$, and that $h'(x)\notin h'(\del X)$.  By the Bieri--Groves theorem for affinoids (Theorem  \ref{affinoid Bieri-Groves theorem}), the tropical fan of $h'(X)$ at $\omega \coloneqq h'(x)$ has pure dimension $d$.

					Let $\Omega \coloneqq h'(X)\setminus h'(\partial X)$; this is an open neighbourhood of $\omega$ in $h'(X)$. For any positive Lagerberg form $\eta \in A^{d-1,d-1}(\Omega)$ with compact support in $\Omega$, we have 
					\begin{equation}\label{current identity} \langle \eta, \d'\d''[F] \rangle = \langle (h')^*(\eta), \d'\d''[h] \rangle \geq 0
					\end{equation} 
					because $h$ is psh.  
					On the other hand, Chambert-Loir and Ducros \cite[Th\'eor\`eme 3.6.1]{chambert_ducros12:forms_courants} have shown that $h'(X)$ is a $d$-dimensional tropical cycle---that is, it is \emph{balanced}---on the neighbourhood $\Omega$ of $\omega$.  
					It follows from \eqref{current identity} that $F$ is a weakly tropically psh piecewise smooth function on $\Omega$. By the tropical maximum principle  shown in Proposition \ref{tropical maximum principle}, the function $F$ is constant in a neighbourhood of $\omega$. Since $h=F \circ h'$, this proves the maximum principle for $h$.
				\end{proof}

\bibliographystyle{egabibstyle}
\bibliography{papers}

\end{document}